\documentclass{LMCS}

\def\doi{9(1:02)2013}
\lmcsheading%
{\doi}
{1--30}
{}
{}
{Sep.~27, 2010}
{Feb.~15, 2013}
{}

\usepackage{amsmath}
\usepackage{amssymb}
\usepackage{tikz}

\usepackage[english]{babel}
\usepackage{amssymb,amsmath,amsthm}
\usepackage{enumerate,hyperref}

\usepackage{url}
\usepackage{graphicx}

\newcommand{\comment}[1]{   }

\newcommand{\subterm}[4]{%
  \fill[draw=gray,fill=gray!20!white,very thick] #1 +(0,#3) --%
  +(#2,-#3) -- +(-#2,-#3) -- cycle; \draw #1+(0,-0.2) node
  [circle] {#4}; }
\newcommand{\subtermspec}[5]{%
  \fill[draw=gray,fill=gray!20!white,very thick] #1 +(0,#3) --%
  +(#2,-#3) -- +(-#2,-#3) -- cycle; \draw #1 +#5 node
  [circle] {#4}; }

\newcommand\lmbda[4]{%
  \fill[black] #1 circle (3pt);
  \draw #1 node [circle, anchor=#4] {#3} -- #2;
}

\newcommand\app[3]{%
  \fill[black] #1 circle (3pt);
  \draw #1 -- #2;
  \draw #1 -- #3;
}

\newcommand\vari[3]{%
  \draw #1 -- #2;
  \draw #2+(0,.2) node [circle, anchor=north] {#3};
}

\theoremstyle{plain}
\newtheorem{theorem}{Theorem}
\newtheorem{main_theorem}[thm]{Main Theorem}

\newtheorem{observation}[theorem]{Observation}

\theoremstyle{definition}

\newcommand{\set}[1]{\left\{#1\right\}}

\newcommand{\liminfty}[1]{\lim_{n \rightarrow \infty} #1}

\DeclareMathOperator{\size}{size}

\newcommand{\Nat}{\mathbb{N}}

\newcommand\correc[1]{#1}%{\color{red} #1}}
\newcommand\correct[1]{#1}%{\color{red} #1}}

\newcommand\corr[1]{#1}%{{\color{red} #1}}

\def\SN{{\ifmmode{SN}\else{$SN$}\fi}}

\begin{document}

\title{
Asymptotically almost all $\lambda$-terms are strongly normalizing
%Some properties of random $\lambda$-terms
}

\author[R.~David]{Ren\'{e} David\rsuper a}
\author[K.~Grygiel]{Katarzyna Grygiel\rsuper b}
\author[J.~Kozik]{Jakub Kozik\rsuper c}
\author[C.~Raffalli]{Christophe Raffalli\rsuper d}
\author[G.~Theyssier]{Guillaume Theyssier\rsuper e}
\author[M.~Zaionc]{Marek Zaionc\rsuper f}
\address{{\lsuper{a,d,e}}LAMA, CNRS,
 Universit\'{e} de Savoie,
 73376 Le Bourget-du-Lac, France}
\email{\{rene.david, christophe.raffalli, guillaume.theyssier\}@univ-savoie.fr}

\address{{\lsuper{b,c,f}}Theoretical Computer Science, Jagiellonian University,
 {\L}ojasiewicza 6, Krak\'{o}w, Poland}
\email{\{Katarzyna.Grygiel, Jakub.Kozik, zaionc\}@tcs.uj.edu.pl}
\thanks{This work was supported
 by  the research project funded by the French Rh\^one-Alpes region and initiated by Pierre Lescanne and by grant number N206 376137 funded by Polish Ministry of Science and Higher Education
}

\begin{abstract}
  We present \correc{a} quantitative analysis of various (syntactic and behavioral) properties of random   $\lambda$-terms. Our main results show that asymptotically\correct{,} almost all terms
  are strongly normalizing and that any fixed closed term almost never appears in a random term.
  Surprisingly, in combinatory logic (the
  translation of the $\lambda$-calculus into combinators), the result
  is exactly opposite. We show that almost all terms are {\em not} strongly
  normalizing. This is due to the fact that any fixed combinator almost always appears in a random combinator.
\end{abstract}
%  Two kinds of technics are used: first, basic
%  combinatorics and  analysis to obtain the density of some sets of terms and, second,
%  coding of terms by
%  other terms having a special form and known density.

\ACMCCS{[{\bf Mathematics of Computating}]: Discrete
  Mathematics---Combinatorics---Combi\-natoric problems}

\maketitle

\section{Introduction}
Since the pioneering work of Church, Turing \textit{et al.}, more
than 70 years ago, a wide range of computational models has been
introduced. 
\correct{It has been shown}
 that \correc{the feasible models} are all equivalent in \correc{the} sense of
computational power.
However, this equivalence says nothing about
what  \emph{typical} programs or machines of each of these
models do.

This paper addresses the following question. Having a
theoretical programming language and a property, what is the
probability that a random program satisfies the given property? In
particular, is it true that almost every random program
satisfies the desired property?
% i.e.\ the probability is 1?
% \comment{The notion of random program
%  is precisely defined in Section \ref{results}.}

We concentrate on functional programming languages and, more
  specifically, on the $\lambda$-calculus, the simplest language \correct{of this kind}
(see \cite{hamkins, rybalov, theyssier} for similar work
on other models of computation). \corr{To our knowledge,} the only work 
on this subject is some experiments \correc{carried out} by Jue Wang (see
\cite{wang}). Most interesting properties of $\lambda$-terms are those
concerning their behavior. However, to analyze them, one has to
consider some syntactic properties as well.

%!!!!
%Various properties
%can be studied. Some concern the structure of a term, some concern
%its behavior.

As far as we know, no asymptotic value for the number of
$\lambda$-terms of size $n$ is known.  We give upper and lower bounds
for this super-exponential number (see Section
\ref{calculus}). Although the gap between the lower and the upper
bound is big (exponential), these estimations are sufficient for our
purpose.

We prove several %his paper proves some  non trivial
results \corr{on} the structural form of a \corr{random} $\lambda$-term. In particular,
we show that almost every closed $\lambda$-term begins with
``many'' lambdas (the precise meaning  is given in Theorem
\ref{th:starting_lambdas}). Moreover, each of them binds  ``many''
occurrences of variables (Theorems \ref{binding1}, \ref{binding2}
and \ref{binding3}). Finally, given any fixed closed
$\lambda$-term, almost no $\lambda$-term has this term as a
subterm (Theorem \ref{avoidaux}).

We also give \corr{results on} the behavior of terms, \correct{which is} our original
motivation. We show that a random term is strongly normalizing
(\SN{} for short) with asymptotic probability  $1$.
\corr{Let us recall} that, in general, \corr{knowing whether a term is} \SN{} is an
undecidable question.

Combinatory logic is another programming language related to the
$\lambda$-calculus. It can be seen as an encoding of
$\lambda$-calculus into a language
%It is one way of coding $\lambda$-terms
without variable binding.
%  which is  fair for questions we are concerned with. This means that
Moreover, there are translations, in both
directions, which preserve the property of being
\SN{}. Surprisingly, our results concerning random combinators are
very different from those for the $\lambda$-calculus.
%We have also studied this language and, surprisingly, the results are very different from those for the $\lambda$-calculus.
For example, we show that for every fixed term $t_0$, almost every
term has $t_0$ as \correc{a} subterm. \corr{This} implies that almost
every term is not \SN{}.
%Note that the increase of size in
%widely-used translations from the $\lambda$-calculus to combinatory logic
%is not known (we only have a lower bound).
The difference of results concerning strong normalization between
$\lambda$-calculus and combinatory logic \corr{is not contradictory
  since the coding of bound variables in combinatory logic induces a
large increase of size}. This is discussed in Section \ref{size}.

\bigskip

%------------------------------------previous work
Our interest in statistical properties of computational objects, like $\lambda$-terms or combinators, is a natural extension \correc{of similar work} on logical objects like formulas or proofs.
This paper is a continuation of the research in which
we try to estimate the properties of random formulas in various logics (especially  the probability of truth, or satisfiability, \corr{of} random formulas).
For the purely implicational logic with one variable (and % at the same time
simple type systems), the exact value of \corr{the} density of true formulas \correc{has} been  computed in % the paper
% of Moczurad, Tyszkiewicz and Zaionc
\cite{mtz00,zaionc05}.
Quantitative relationship between intuitionistic and classical logics (based on the same language) has also been analyzed. The exact value describing how \correc{large the intuitionistic} fragment of the classical logic with one variable \correct{is} has been determined in % Kostrzycka and Zaionc
\cite{kos-zaionc03}. For results with more \correc{than} one variable, \correc{or with} other logical connectives, consult \cite{FGGZ,GKZ,GK-09}.

 \comment{The case of and/or connectors received much attention -- see Lefmann and Savick\'{y} \cite{LS97},
Chauvin, Flajolet, Gardy and Gittenberger \cite{CFGG04}
and Gardy and Woods \cite{GW05}. We refer to Gardy \cite{gardy-dmtcs} for a survey on
probability distribution on Boolean functions induced by random Boolean
expressions.
}

\bigskip

\correct{The organization of the paper is as follows.
In Section \ref{lambda} we recall basic definitions and facts
about $\lambda$-calculus and combinatory logic.
%and new concept of $\lambda$-trees which are finite graph
%theoretical objects which can be enumerated.
%In this section we recall also the concept of combinatory logic.
Section \ref{maths}~gives
combinatorial notations which we will need in our proofs.
It introduces generating functions
and basic techniques to compute asymptotics.
The notion of density and its basic properties is introduced in Section \ref{densities}.
The lower and
upper bounds for the number of $\lambda$-terms of size $n$ are given in
Section \ref{calculus}.
In Section \ref{coding}~we
prove theorems about random $\lambda$-terms using coding which is
an injective and size-preserving function on terms.
Our main result establishing that \corr{the}
set of strongly normalizable terms has density $1$
appears at the end of this section in Theorem \ref{main}.
Section \ref{cl} contains results in combinatory
logic, namely the fact that every fixed term appears in almost every term.
The main result of this section, \corr{in Theorem \ref{main_CL}, states} that
the density of non-strongly normalizing combinators is $1$.
Finally Section \ref{size} discusses future work, open questions and possible applications of results.
} 

\section{$\lambda$-calculus and combinatory logic}\label{lambda}

\subsection{$\lambda$-calculus}

We start with presenting some fundamental concepts of the
$\lambda$-calculus, as well as with some new definitions used in this
paper.  \correc{We do this mainly to make our notations and
  conventions precise. It should be enough for defining the notion of
  size, but for substitution and reduction and normalization we
  recommend \cite{BAR84}.}

\begin{defi}
Let $V$ be a countable set of variables. The set $\overline{\Lambda}$ of $\lambda$-terms is defined by the following grammar:
$$t := \ V \ \mid \ \lambda V. t \ \mid \ (t \ t)$$

We denote by $\Lambda$ the set of all closed $\lambda$-terms. We write
$t_1 \ t_2 \dots t_n$ without \correc{parentheses} for $(\dots(t_1 \ t_2) \dots t_n)$.
\end{defi}

As usual, $\lambda$-terms are considered modulo % the
$\alpha$-equivalence, i.e.\ two terms which differ only by the names of bound variables are considered equal.

Let us observe that $\lambda$-terms can be seen as rooted unary-binary trees.

\begin{defi}
By \correc{a} $\lambda$-tree we mean \correc{a} rooted tree \correc{of the following form} there are two kinds of \correc{inner} nodes -- labeled with $@$ and with $\lambda$. Nodes labeled with  $@$ have two successors\correc{:} left and right. Nodes labeled with $\lambda$ have only one successor. Each Leaf of a tree  is labeled \correc{either} \correct{with} \correc{a} variable or \correct{with} \correc{a pointer to one of the} $\lambda$ nodes above it.
\end{defi}

%\begin{defi}
\correct{
\corr{For} every $\lambda$-term $t$ we define the $\lambda$-tree $G(t)$ in the following way:
\begin{enumerate}[$\bullet$]
  \item If $t$ is a variable $x$, then $G(t)$ is a single node labeled with $x$.
  \item If $t= t_1 t_2$, then $G( \correc{t_1} \correc{t_2}) $ is a tree with the root labeled with $@$ and two subtrees $G(\correc{t_1})$ (left) and $G(\correc{t_2})$ (right).
  \item If $t= \lambda x.u$, then  $G(t)$ is obtained from $G(u)$ in four steps:
   \begin{enumerate}[$-$]
     \item add a new root labeled with $\lambda$;
     \item connect the new root with $G(\correc{u})$;
     \item connect all leaves of $G(\correc{u})$ labeled with $x$ with the new root;
     \item remove all labels $x$.
   \end{enumerate}

\end{enumerate}
}
%\end{defi}

\bigskip

\bigskip

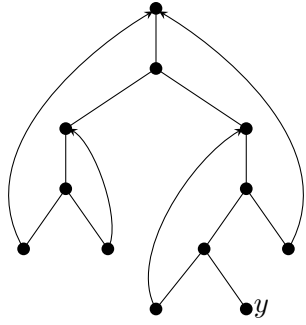
\begin{figure}[ht]
  \centering
  \begin{tikzpicture}[>=stealth,scale=.8]
    \lmbda{(0,5)}{(0,4)}{}{east}
    \app{(0,4)}{(-1.5,3)}{(1.5,3)}
    \lmbda{(-1.5,3)}{(-1.5,2)}{}{east}
    \app{(-1.5,2)}{(-2.2,1)}{(-0.8,1)}
    \lmbda{(1.5,3)}{(1.5,2)}{}{east}
    \app{(1.5,2)}{(0.8,1)}{(2.2,1)}
    \app{(0.8,1)}{(0,0)}{(1.5,0)}
    \fill[black] (-2.2,1) circle (3pt)
                 (-0.8,1) circle (3pt)
                 (0,0) circle (3pt)
                 (1.5,0) circle (3pt)
                 (2.2,1) circle (3pt);
   \draw [->] (-2.2,1) .. controls +(120:2cm) and +(-150:1cm) .. (-0.05,4.95);
   \draw [->] (-0.8,1) .. controls +(60:0.5cm) and +(-30:0.5cm) .. (-1.45,3);
   \draw [->] (0,0) .. controls +(120:1cm) and +(-150:1cm) .. (1.45,3);
   \draw [->] (2.2,1) .. controls +(60:2cm) and +(-30:1cm) .. (0.05,4.95);
   \draw (1.75,0) node {$y$};
  \end{tikzpicture}
\caption{\rm The $\lambda$-tree representing the term $\lambda z.(\lambda u. zu)\correc{(\lambda u. uyz)}$ (labels of inner nodes are not shown in the figure and can be recovered from their degrees)}
%\correc{(labels of inner nodes are not represented since they can be recovered from the number of successors of the node)}
\label{fig:headlambdaterm_a}
\end{figure}

%For example $\lambda$-tree above is a tree for the term $\lambda z. (\lambda u . zu)((\lambda u. uy)z)$.
%We can proof simple facts about $\lambda$-trees.

\begin{observation}
 If $T$ is a $\lambda$-tree then $T=G(t)$ for some $\lambda$-term $t$.
 Terms  $t$ and $u$ are $\alpha$-convertible iff  $G(t)$ \correc{and} $G(u)$ \correc{are the same tree}.
 \end{observation}

\comment{
For every term $t$, if we forget about variable binding we obtain a unary-binary tree. We call it \emph{the structure} of $t$. Removing from the structure the unary nodes and connecting binary ones and leaves so that to preserve the original connectivity, we obtain \emph{binary structure} of $t$.
}

We often use (without giving the precise definition) the classical terminology about trees (e.g. path, root, leaf, etc.). A path from the root to a leaf is called a branch.

\begin{defi} Let $t$ be a $\lambda$-term.
\begin{enumerate}[(1)]
\item  A term $t'$ is a \emph{subterm} of $t$ (denoted as $t' \leq t$) if
\begin{enumerate}[$-$]
    \item either $t=t'$,
    \item or $t=\lambda x.u$ and $t' \leq u$,
    \item or $t=(u \ v)$ and $(t'\leq u$ or $t'\leq v)$.
\end{enumerate}
\item Let $u=\lambda x.a$ be a subterm of $t$. We say that this occurrence of $\lambda x$ is \emph{binding} in $t$ if $x$ has a free occurrence in $a$.
\item The \emph{unary height} of $t$ is the maximum number of lambdas on a \correc{branch} % (a path from the root to some leaf)
  in \correc{the} $\lambda$-tree of $t$.
\item Two lambdas in $t$ are called \emph{incomparable} if there is no branch in the $\lambda$-tree containing both of them. The $\lambda$-\emph{width}  of $t$ \correc{(or simply width of $t$ when there is no ambiguity)} is the maximum number of pairwise incomparable binding lambdas. \textit{Remark: }a closed $\lambda$-term has width at least $1$.
\item We say that $t$ has $k$ \emph{head lambdas} if its $\lambda$-tree  starts with at least $k$ unary nodes.
\end{enumerate}
\end{defi}

\begin{defi}\ %\hbox{1em}
\begin{enumerate}[$\bullet$]
  \item \correc{When $t$ and $u$ are terms, $t[x:=u]$ denotes the {\em capture
    avoiding substitution} of $u$ for the free occurrences of the
  variable $x$ in $t$. Bound variables of $t$ may have to be renamed to avoid
  capture of free variables in $u$.}
  \item A term of the form $(\lambda x.t)u$ is called a \emph{$\beta$-redex}. A
$\lambda$-term is in \emph{normal form} if it does not contain $\beta$-redex
\corr{subterms}. The least relation $\triangleright$ on terms satisfying
$ (\lambda x.t)u \triangleright t[x:=u]$ and closed \correc{under} context\correc{s} is
called \emph{$\beta$-reduction}.
  \item  A term $t$ is \emph{(weakly) normalizing} if there is \correc{a} finite
reduction sequence starting from $t$ and ending in a normal form.
  \item
 A term $t$ is \emph{strongly normalizing} \correc{(SN)} if all reduction sequences \correc{starting from $t$} are finite. \correc{If $t$ is \SN{}, we denote by $\eta(t)$ the length of its longest reduction}. The fact that such a longest reduction exists follows from K\"{o}nig's lemma. If $t$ is not
\SN{}, $\eta(t) = +\infty$.
\end{enumerate}

\end{defi}

% As an example we can see
\noindent\correc{In the $\lambda$-tree representation, a redex is a subtree of the $\lambda$-tree.}
Therefore $\beta$-reduction can be seen as \correc{an} operation on $\lambda$-trees (see Fig. \ref{fig:headlambdaterm_b}).

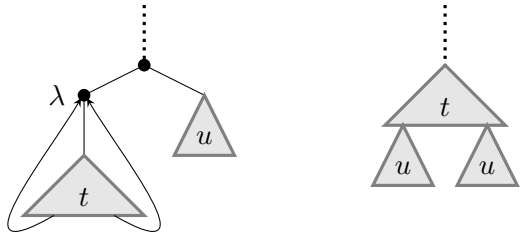
\begin{figure}[ht]
  \centering
  \begin{tikzpicture}[>=stealth,scale=.8]
    \draw[dotted, very thick] (-3,3) -- (-3,2);
    \app{(-3,2)}{(-4,1.5)}{(-2,1.5)}
    \lmbda{(-4,1.5)}{(-4,0.5)} {$\lambda$}{east}
    \subterm{(-4,0)}{1}{.5}{$t$}
    \subterm{(-2,1)}{.5}{.5}{$u$}
    \draw [->] (-4.5,-0.5) .. controls +(-150:2cm) and +(-120:1cm) .. (-4.05,1.45);
    \draw [->] (-3.5,-0.5) .. controls +(-30:2cm) and +(-60:1cm) .. (-3.95,1.45);
  \draw[dotted, very thick] (2,3) -- (2,2);
  \subterm{(2,1.5)}{1}{.5}{$t$}
  \subterm{(1.3,0.5)}{.5}{.5}{$u$}
  \subterm{(2.7,0.5)}{.5}{.5}{$u$}
  \end{tikzpicture}
\caption{\rm $\beta$-reduction scheme}
\label{fig:headlambdaterm_b}
\end{figure}

\begin{defi}
The \emph{size}  of a  term (denoted by $\size(\cdot)$) is defined recursively as follows:
\begin{enumerate}[(i)]
\item $\size(x) = 0$ if $x$ is a variable,
\item $\size(\lambda x.t)= 1 + \size(t)$,
\item $\size(t \ u)= 1 + \size(t) + \size(u)$.
\end{enumerate}
\end{defi}

As we can see, $\size (t) $ is the number of \correc{inner} nodes in \correc{the} $\lambda$-tree $G(t)$.

\begin{nota}
Let $n$ be an integer. We denote by $\Lambda_n$ the set of closed terms of size $n$. Obviously, the set $\Lambda_n$ is finite. We denote its cardinality by $L_n$.
\end{nota}

As far as we know, no asymptotic analysis of the sequence $\correc{\bigl(L_n\bigr)_{n\in\Nat}}$ has been done. Moreover, typical combinatorial techniques do not seem to apply easily for this task.

\subsection{Innocuous and safe $\lambda$-terms}

This sections introduces the notion of \emph{safe $\lambda$-terms} which is a sufficient condition for being SN (Proposition~\ref{width2sn}).
% Nevertheless,

\begin{defi}[ ]~
\begin{enumerate}[(1)]
\item Let $t$ be a term of width $1$. We say that $t$ is {\em innocuous} if there is no binding $\lambda$ on the leftmost branch of $t$ (this includes the root of $t$).
\item \correc{We say that $t$ is {\em safe} if either it has width at most $1$ or if
  it has width $2$ and for $(u \ v)$ being the
  smallest subterm of $t$ of width $2$, at least one of the terms $u$ and $v$ is innocuous.}  
\end{enumerate}
\end{defi}

\begin{defi}~
\begin{enumerate}[$\bullet$]
\item A \emph{substitution} $\sigma$ is a partial map from variables to
terms such that the domain of $\sigma$ is finite. Let $t$ be a term and $\sigma$ be a substitution. By $t[\sigma]$ we denote the term obtained from $t$ by simultaneous replacement of all free occurrences of variables $x$ from the domain of $\sigma$ by $\sigma(x)$.
\item A context is a $\lambda$-term with a unique hole denoted by $[]$. Traditionally, contexts are defined by a BNF grammar:
$$E:=[] \ | \  \lambda x.E \ | \ (E \  \correc{\overline\Lambda}) \ | \ (\correc{\overline\Lambda} \ E) \ \textrm{where $\correc{\overline\Lambda}$ \corr{denotes} arbitrary terms.}$$
\item When $E$ is a context and $t$ is a term, $E[t]$ denotes the result of replacing  the hole in $E$ by $t$ allowing captures (i.e.  the
lambdas in $E$ can bind variables in $t$).
\item For a context $E$,we define  $\eta(E)$ as  $\eta(E[x])$ and $\size(E)$ as $\size(E[x])$ where $x$ is an arbitrary variable not captured by $E$.
\item In a few cases, we need contexts with multiple holes. When $E$ is a context with exactly $n$ holes\correc{, } $E[t_1,\dots,t_n]$ denotes the term where the holes of
  $E$ are substituted from the leftmost to the rightmost by terms
  $t_1,\dots,t_n$ (in this order).
\end{enumerate}
\end{defi}

\correct{
\noindent In some proofs in this section we use the following basic fact concerning strong normalization of $\lambda$-terms:

\begin{fact}\label{lambdabasic}
Let $t$ be a $\lambda$-term.
\begin{enumerate}[$\bullet$]
\item If $t = (x\,t_1\,\dots\,t_n)$, for some  variable $x$, with $n \geq 0$, then $\eta(t) = \eta(t_1) + \dots + \eta(t_n)$. Moreover $t$ is $\SN{}$ if and only
  if $t_1,\dots,t_n$ are \SN{}.
\item If $t = \lambda x.u$, then $\eta(t) = \eta(u)$ and $t$ is $\SN{}$ if and only if $u$ is \SN{}.
\item If $t = ((\lambda x.u)\,v\,t_1\,\dots\,t_n)$ with $n \geq 0$ and $t$ is $\SN{}$, then $\eta(t) > \eta(u[x:=v]\,t_1\,\dots\,t_n)$ and  $\eta(t) > \eta(v) + \eta(t_1) + \dots + \eta(t_n)$. Moreover $t$ is $\SN{}$ if and only if $v$ and $(u[x:=v]\,t_1\,\dots\,t_n)$ are \SN{}.
\end{enumerate}
These three cases cover all possible forms of $t$. Moreover, if $x$ is a variable, then $t$ is \SN{} if and only if $(t\,x)$ is \SN{}.
\end{fact}
}
\proof
This facts are ``folklore'', but they are not trivial to
prove directly from the definition of $\beta$-reduction and the proof 
is not found in the usual litterature. Here, we give a proof sketch
using the fact that Barendregt's \cite{BAR84}
{\em perpetual norm} (length of the {\em perpetual reduction strategy}) is in fact
the length of the longest reduction. This is proved in
\cite{regnier}.

The perpetual strategy is the strategy that reduces the left-most
redex first, except when this redex is a K-redex ($(\lambda x.u) t$
when $x$ is not free in $u$). In this case, the redex is reduced only
when $t$ and $u$ are normal. For a formal definition see \cite{BAR84}
or \cite{regnier}.

The equality about $\eta(t)$ in the first two items are immediate from
this, by induction on the length of the reduction.

Using the perpetual norm, we have
$$\eta((\lambda x.u)\,v\,t_1\,\dots\,t_n) = 1 +
max(\eta(u[x:=v]\,t_1\,\dots\,t_n), \eta(v) + \eta(u
\,t_1\,\dots\,t_n)).$$ The two terms in the max correspond
respectively to the case where $x$ occurs free in $u$ and the case where
the redex is a K-redex.

For the equivalence, one direction comes from the fact that subterms
and reducts of an $\SN{}$ term are  $\SN{}$. For the other direction we have to prove that if
$v$ and $(u[x:=v]\,t_1\,\dots\,t_n)$ are $\SN{}$ then so is
$t=((\lambda x.u)\,v\,t_1\,\dots\,t_n)$. This is done by induction on
$\eta(u)+\eta(v)  + \eta(t_1) + \dots + \eta(t_n)$ looking at the
different possible reductions of $t$.
 
The fact that if $t$ is \SN{} then so is $(t \ x)$ is proved using
the perpetual norm to establish that $\eta(t \ x) \leq \eta(t) + 1$
(in fact $\eta(t \ x) = \eta(t) + 1$ if $t$ reduces to a term starting
with $\lambda$ and $\eta(t \ x) = \eta(t)$ otherwise).
\qed

%\noindent {\bf Remark}
%
%The previous fact is well-known (see for example \cite{BAR84}).
%Also note that, if $t$ is a term and $x$ is a variable, then $t$ is \SN{} if and only if $(t\,x)$ is \SN{}.

\begin{lem}\label{sn_l1}
The set of terms of width at most $1$ is closed under $\beta$-reduction.
\end{lem}

\proof
If a term is of width $0$, then \correc{no reduction can change the
width, since width $0$ just means that all variables in the term are free.}

Let $t$ be a term of width $1$. First, let us remark that all binding lambdas in $t$ occur on the same branch. We consider a $\beta$-reduction:
\[t = E[(\lambda x.u)\,v] \triangleright E[u[x:=v]] = t'.\] There are
two cases: either $x$ has no \correc{free} occurrences in $u$ and $t' = E[u]$ or it has some \correc{free} occurrence in $u$ and $v$ must have width $0$, which means that
\correc{every variable of $v$ is either free in $t$ or bound by some lambda occurring in the context $E$}. It is clear that $t'$ is still of width
$1$ because the binding lambdas remain on one branch.
\qed

\begin{lem}\label{sn_l2}
If $t$ is a term of lambda width at most $1$, then $t$ is \SN{}.
\end{lem}

\proof
Let $N_0(t)$ and $N_1(t)$ denote the number of, respectively,
non-binding and binding lambdas in term $t$. Let us introduce the
lexicographic order on pairs $\langle N_1(t), N_0(t)\rangle$. Let $t$
be of width at most $1$. Then, performing a $\beta$-reduction on $t$
decreases the pair $\langle N_1(t), N_0(t)\rangle$ while keeping the
width at most $1$ by Lemma \ref{sn_l1}.
To prove this, we consider a $\beta$-reduction: $t = E[(\lambda
x.u)\,v] \triangleright E[u[x:=v]] = t'$ and distinguish two cases:
\begin{enumerate}[$\bullet$]
\item \correc{If $x$ does not
  occur in $u$,
then $N_1(t)$ is non-increasing. Moreover, it is decreasing if $v$ contains some
binding lambdas or if $E$ \corr{binds} some variables that occur only in
$v$. Therefore, if $N_1(t)$ is constant, then $N_0(t)$ is decreasing: we erase at least one
non-binding $\lambda$ and do not transform binding ones \corr{into} non-binding
ones.} 
\item If $x$ occurs in $u$, then $v$ is of width
$0$ and contains no binding $\lambda$, which means that we erase one
binding $\lambda$ and only duplicate non-binding
lambdas. Therefore, $N_1(t)$ is decreasing.\qed
\end{enumerate}

\begin{lem}\label{width0fun}
If $u$ has width $0$ and $t_1,\dots,t_n$ are \SN{} terms, then the
term $t= (u\,t_1\,\dots\,t_n)$ is \SN{}.
\end{lem}

\proof
By induction on \correc{the} size of $u$. We distinguish three cases:
\begin{enumerate}[$\bullet$]
\item If $u = x$, the result is trivial by Fact \ref{lambdabasic}.
\item If $u = (u'\,v)$, $v$ has width $0$ and is \SN{} because of Lemma
  \ref{sn_l2}. We conclude by induction on $u'$.
\item For $u = \lambda x.u'$ we consider two cases: if $n = 0$, the result follows from Lemma \ref{sn_l2}; otherwise,
 by Fact \ref{lambdabasic}, it is enough to show that the head
 reduct of $t$ is \SN{}. But, since $u$ has width $0$, this reduct is
 $(u'\,t_2\,\dots\,t_n)$ and the result follows \correct{from the
 induction hypothesis}. \qed
\end{enumerate}

\begin{lem}\label{sn_l3}
  Let $t \in \SN{}$ be a term and $\sigma$ be a substitution such that, for each $x$, there is $k$ such that $\sigma(x) = (u\,v_1\,\dots\,v_k)$ where $u$ has width $0$ and $v_1$\,\dots\,$v_k$ are \SN{}. Then $t[\sigma] \in \SN{}$.
\end{lem}

\proof
By induction on $\langle\eta(t), \size(t)\rangle$ ordered lexicographically. We consider the following cases:
\begin{enumerate}[$\bullet$]
\item If $t = \lambda x.t_1$ or if $t = (x\,t_1\,\dots\,t_n)$ with $x$ not in the domain of $\sigma$, it is enough to prove that for all $i$, $t_i[\sigma]$ is \SN{}. This follows from the induction hypothesis because $\eta(t_i) \leq \eta(t)$ and $\size(t_i) < \size(t)$.
\item If $t = ((\lambda x.u)\,v\,t_1\,\dots\,t_n)$ we show that $v[\sigma]$ and $(u[x:=v]\,t_1\,\dots\,t_n)[\sigma]$ are \SN{} and apply Fact \ref{lambdabasic}.  This follows from the induction hypothesis
because \correc{$\eta(v) < \eta(t)$} for the first point and because
$\eta(u[x:=v]\,t_1\,\dots\,t_n) < \eta(t)$ for the second.
\item If $t = (x\,t_1\,\dots\,t_n)$ where $x$ is in the domain of
  $\sigma$. Then we have $t[\sigma]\! =\!
  (\sigma(x)\,t_1[\sigma]\,\dots\,t_n[\sigma])$ which is \SN{} by Lemma
  \ref{width0fun} because $t_1[\sigma], \dots, t_n[\sigma]$ are \SN{} by the induction hypothesis and $\sigma(x) = (u\,v_1\,\dots\,v_k)$ where $u$ has width $0$ and $v_1$\,\dots\,$v_k$ are \SN{}. \qed
\end{enumerate}

\begin{defi}
We define the set of contexts of width \correc{at most} $1$ by the following BNF grammar (where $\correc{\overline\Lambda_0}$ denotes the set of $\lambda$-terms of width $0$):
\[ E := [] \ | \ \lambda x. E \ | \ (E \ \correc{\overline\Lambda_0}) \ | \ (\correc{\overline\Lambda_0} \ E).\]
\end{defi}\medskip

\noindent This definition means that all the binding lambdas are on the path from the root to the hole of the context. 

\begin{lem}\label{sn_l4}
Let $E$ be a context of width $1$ and $u \in \SN{}$ be a term. Then $E[u] \in \SN{}$.
\end{lem}

\proof
By induction on $\size(E)$. Cases $E=[]$ or $E=\lambda x. E_1$ are trivial (in the second case, since $\size(E_1) < \size(E)$, the proof goes by the induction hypothesis).

If $E=(E_1 \ v)$, where $v \in \corr{\overline\Lambda_0}$, then $E[u]=(E_1[u] \ x)[x:=v]$ where $x$ is a fresh variable. $E_1[u]$ is \SN{} by induction hypothesis because $\size(E_1) < \size(E)$. Therefore $(E_1[u] \ x)$ is \SN{} by Fact \ref{lambdabasic} and finally $(E_1[u] \ x)[x:=v]$ is \SN{} by Lemma \ref{sn_l3}.

\correc{If $E=(v \ E_1)$, then $E[u]=(x\ E_1[u])[x:=v]$ where $x$ is a fresh
variable and $E_1[u]$ is \SN{} by induction hypothesis because
$\size(E_1) < \size(E)$. Therefore $(x\ E_1[u])$ is \SN{} and finally $(x\ E_1[u])[x:=v]$ is \SN{} by Lemma \ref{sn_l3}.}
\qed

\begin{prop}\label{width2sn}
\correc{All safe terms are \SN{}.}
\end{prop}

\proof
If $t$ has width at most one, the result follows directly from
Lemma \ref{sn_l2}. If $t$ has width $2$, let $(t_1 \ t_2)$ be the
smallest subterm of $t$ of width $2$. This means that $t$ can be
written \correc{as} $E[(t_1 \ t_2)]$ where $E$ is a context of width at most $1$ and
$t_1$ and $t_2$ are both of width $1$. By Lemma \ref{sn_l4}, it is
therefore enough to show that $(t_1 \ t_2)$ is \SN{}.

We know that $t$ is safe. This means that at least $t_1$ or $t_2$ is
innocuous.  If $t_i$ is innocuous, it can be written $F[(u \ v)]$ where $u$
has width $0$, $v$ has width $1$ and $F$ belongs to the family of contexts defined by
the following BNF grammar:
\[ F :=  [] \ | \ \lambda \_. F  \ | \ (F \ \corr{\overline\Lambda_0})\]
 where $\lambda \_$ denotes non-binding lambdas and $\corr{\overline\Lambda_0}$ denotes  terms of width 0.

The context $F$ is defined precisely to denote the beginning of the leftmost branch
until we reach an application node whose argument is of width $1$. The
definition of innocuous \corr{terms} together with the definition of width $1$ ensures the existence of such an
application node on the leftmost branch.

This means that $(t_1 \ t_2)$ can be written $(F[(u \ v)] \ t_2)$
($(t_1 \ F[(u \ v)])$ resp.). Let us define $t' = (F[x] \ t_2)$
(resp. $t' = (t_1 \ F[x])$), for a fresh variable $x$.

In both cases, $(t_1 \ t_2) = t'[x:=(u \ v)]$. We can conclude by
Lemma \ref{sn_l3} \corr{that $(t_1\ t_2)$} is \SN{} since $u$ has width $0$ and  $t'$ and
$v$ are \SN{} (by Lemma \ref{sn_l2}, since they have width $1$).
\qed

\subsection{Combinatory logic}\label{Combinatory Logic}

Combinatory logic is a theoretical model of computation introduced by
Moses Sch\"{o}nfinkel in \cite{schonfinkel} and many years later
rediscovered and deeply studied by Haskell Curry in
\cite{curry_feys}. For the main reference \correct{to} the subject we
refer to %propose Barendregt
\cite{BAR84}.
A very intelligible approach towards this subject can be found in \cite{smullyan}.
It is a well known fact that both models, the lambda calculus and the combinatory logic, are equivalent in \correc{the} sense  of \corr{expressive} power. It turns out, however, that \corr{these} two models differ radically \corr{as regards} the behavior of random terms.

\begin{defi}{Combinatory logic}
\begin{enumerate}[(1)]
\item The set $\mathcal{F}$ of combinatory terms, {\it combinators}, is defined by the following grammar:
$${\mathcal F}:=  K  \; \mid \;  S  \; \mid \;  I \; \mid  ({\mathcal F} \; {\mathcal F}).$$
The notational conventions concerning parentheses are the same \corr{as} for
$\lambda$-terms \correct{i.e.\ } \corr{we write}  $t_1 \ t_2 \dots t_n$ without parentheses for $(\dots(t_1 \ t_2) \dots t_n)$.
\item  \correc{The {\it reduction} on combinators is the least
  compatible relation $\triangleright$ satisfying the following rules:}
$$K \; u \; v \rhd u \qquad S \; u\; v \; w \rhd  u \; w \; (v \; w) \qquad I \ u \rhd  u.$$\smallskip
\end{enumerate}
\end{defi}

% Let us observe that
\noindent Combinatory terms can be considered as rooted binary \corr{trees} \correc{whose} leaves are labeled \correct{with} combinators $K, S$ and $I$ and \correc{inner} nodes are labeled \correct{with} an application operation.
% Therefore  every combinatory term $t$ can be associated uniquely with the combinatory tree $G(t)$.
Accordingly, every reduction rule can be seen as a transformation of combinatory \corr{trees}.

\begin{defi}
A combinatory term is in {\it normal form} if no reduction can be performed.
A term $M$ is \emph{normalizing} if there is \correc{a} reduction sequence starting from $M$ and ending in a normal form $N$. A term $M$ is \emph{strongly normalizing} if all reduction sequences are finite. 
% By $\SN{}$ we mean all terms which are strongly normalizing.
\end{defi}

\begin{defi}{Subterm and size}
\begin{enumerate}[(1)]
\item A combinator $u$ is a {\it subterm} of $v$ if either $u=v$ or $v$ is of the form $v_1 \; v_2$ and $u$ is a subterm of $v_1$ or $v_2$.
\item The {\it size} of a combinator is defined by the following rules:
$$\size(S)=\size(K)=\size(I)=0 \quad \text{and} \quad \size(u \ v)=1+\size(u)+\size(v).$$
\end{enumerate}
\end{defi}

As we can see $\size (t) $ is the number of \correc{inner} nodes of the combinatory tree of $t$.

\begin{nota}
For an integer $n$, we denote by $\mathcal{F}_n$ the set of combinatory terms of size $n$. The set $\mathcal{F}_n$ is finite and we denote its cardinality by $F_n$.
\end{nota}
\newcommand{\C}{\mathcal{C}}

\section{Combinatorial results}\label{maths}

The following standard notions will be used throughout the whole paper.

\begin{defi} Let $f,g \colon {\mathbb N} \to {\mathbb R}$.
\begin{enumerate}[(i)]
\item Functions $f$ and $g$ are said to be \emph{asymptotically equal}
 iff $\lim_{n \to \infty} \frac{f(n)}{g(n)}=1$. We denote it by $f \sim g$.
\item The \emph{asymptotic inequality} $f \gtrsim g$
holds iff there exists a function $h \colon {\mathbb N} \to
{\mathbb R}$ such that $h \sim g$ and $f(n) \geq h(n)$ for \corr{all}
$n$.
\item A function $f$ is said to be of the \emph{smaller order}
 than $g$ iff $\lim_{n \to \infty} \frac{f(n)}{g(n)} = 0$. We denote it by $f \in o(g)$.
\item A function $f$ is said to be \emph{subexponential} in $n$ iff
there exists $h \colon {\mathbb N} \to {\mathbb R}$ such that $h
\in o(n)$ and $f(n) = 2^{h(n)}$.
\item If $x$ is a real number we denote by $\lfloor  x  \rfloor$
(resp.\ $\lceil x \rceil$) the largest (resp.\ smallest) integer $n$
such that $n\leq x$ (resp.\ $x \leq n$).
\end{enumerate}
\end{defi}

\noindent {\bf  Notation}

When an unknown function $f$ is, for example, asymptotically equal
to an explicit function (say for example $n \mapsto n\ln(n)$) we will write
$f \sim n\ln(n)$ or sometimes $f(n) \sim n\ln(n)$.

\subsection{Generating function method}

%%%%%%%%%%%%%%%%%%%%%%%%%%%%%%%%%%%%%%%%%%%%%%%%%%%%%%%%%%%%%%%%%%%%%%%%%%%%%%%%%%%%%%%%%%%%%%%%%%%%%
%%%%%%%%%%%%%%%%%%%%%%%%%%%%%%%%%%%%%%%%%%%%%%%%%%%%%%%%%%%%%%%%%%%%%%%%%%%%%%%%%%%%%%%%%%%%%%%%%%%%%
%%%%%%%%%%%%%%%%%%%%%%%%%%%%%%%%%%%%%%%%%%%%%%%%%%%%%%%%%%%%%%%%%%%%%%%%%%%%%%%%%%%%%%%%%%%%%%%%%%%%%
\correct{Many questions concerning the asymptotic behavior of sequences of real positive numbers  can be
efficiently resolved by analyzing the behavior of %$f_{A}$
their generating functions (see \cite{Wilf} for introductory reference).
This is the approach we take to determine the asymptotic fraction
of certain combinatory logic trees of a given size.

The \correc{following} theorem is a well-known result in the theory of generating functions.
 Its derivation from the Szeg\"o Lemma (see \cite{szego}) can be found, e.g., in \cite{zai06} (Theorem 22).} We denote by
 $[z^n]\{v(z)\}$ the coefficient of $z^n$ in the expansion of $v$.

\begin{thm}\label{glowniejsze}
Let $v$, $w$ be functions satisfying the following conditions:
\begin{enumerate}[\em(i)]
  \item $v, w$ are analytic in $|z|<1$ with $z=1$ being the only singularity on the circle $|z|=1$,
  \item $v, w$ have the following expansions in the vicinity of $z=1$:
$$ v(z)=\sum_{p\geq 0}v_p(1-z)^{p/2}, \qquad w(z)=\sum_{p\geq 0}w_p(1-z)^{p/2}$$
where $w_1 \neq 0$.
\end{enumerate}

Let $\widetilde{v}$ and $\widetilde{w}$ be defined by $\widetilde{v}(\sqrt{1-z}) = v(z)$ and $\widetilde{w}(\sqrt{1-z}) = w(z)$. Then
$$ \liminfty{\frac{[z^n]\{v(z)\}}{[z^n]\{w(z)\}}} = \frac{v_1}{w_1} = \frac{(\widetilde{v})'(0)}{(\widetilde{w})'(0)} .$$
\end{thm}

\subsection{Catalan numbers }
We denote by  $C(n)$ the $n$-th Catalan numbers, i.e., the number of binary
trees with $n$ inner nodes. We use the following classical result (see, for example, \cite[Ch. IV.1]{fs01}).
%\todo{Can anyone point any specific chapter, or maybe other reference?}

\begin{prop}\label{catalan}~
\begin{enumerate}[$\bullet$]
\item $C(n+1) = \sum_{i=0}^n C(i)C(n-i)$ for $n > 0$ and
  $C(0)=1$. From this we have $C(n+1) \geq \sum_{i=0}^n C(i)$.
\item $C(n) = \frac{1}{n+1}{2 n\choose n} = \prod_{i=2}^n
  \frac{n+i}{i}$. From this we have $\frac{C(n)}{C(n-1)} = \frac{2
    (2 n - 1)}{n + 1}$

 \item $C(n) \sim \frac{4^n}{n^{3/2}\sqrt{\pi}}$ and thus, for $n$ large enough,  we have\\
 $C(n) \geq \gamma \frac{4^n}{n^{3/2}}$  for some constant $0 < \gamma
 < 1$.
\end{enumerate}
\end{prop}

\subsection{Large Schr\"{o}der numbers}\label{schroder}

We denote by $M(n,k)$ the number of unary-binary trees with $n$
inner nodes and $k$ leaves. Let  $M(n) = \sum_{k \geq 1}
M(n,k)$ denote the number of unary-binary trees with $n$ inner
nodes. These numbers are known as the large Schr\"{o}der
numbers. Note that, since in this paper the size of \corr{variables} is
$0$, we use them instead of the so-called Motzkin numbers
which enumerate unary-binary trees with $n$ inner and outer nodes. We
use the following proposition.

\begin{prop}\label{motzkin}
\begin{enumerate}[$\bullet$]
  \item $M(n,k) = C(k-1){n+k-1 \choose n-k+1}.$
  \item $M(n) \sim \left(\frac{1}{3-2\sqrt{2}}\right)^n \frac{1}{\sqrt{\pi} n^{3/2}}.$
\end{enumerate}

\end{prop}
\proof
(1) Every unary-binary tree with $n$ inner nodes and $k$ leaves has
$k-1$ binary and $n-k+1$ unary nodes. We have $C(k-1)$ binary trees
with $k$ leaves. Every such a tree has $2 k-1$ nodes (inner nodes and leaves). Therefore there are ${n+k-1 \choose n-k+1}$ possibilities of inserting $n-k+1$ unary nodes (we can put \correct{a} unary node above every node of a binary tree).

(2) The asymptotics for $M(n)$ is obtained by using standard tools of the generating function \correc{method} (see, e.g., \cite[Ch.VII.4]{fs01} for exact computations).
%(for this sequence \corr{the generating function} is equal to $m(x)=\frac{1-x-\sqrt{1-6x+x^2}}{2x}$). For more details see \cite[Ch.VII.4]{fs01}.
\qed
\section{Densities}\label{densities}

\subsection{Main notations} \label{notation}
For any finite set $A$ we denote by $\#A$ its \corr{cardinality}.
To attribute a precise meaning to sentences like ``asymptotically \corr{almost} all $\lambda$-terms have property $P$'' we use the following definition of asymptotic density. 

\begin{defi}\label{density}
Let $B \subset \Lambda$, assume that $B$ contains closed terms of every large enough size. For $A \subseteq B$, if the  limit
\[\lim_{n\rightarrow\infty}\frac{\#(A\cap \Lambda_n)}{\#(B \cap \Lambda_n)}\]
exists, then we call it the \emph{asymptotic density} of $A$ in $B$ and denote it by $d_B (A)$.
\end{defi}

\noindent{\bf Remarks and notations}

\begin{enumerate}[$\bullet$]
  \item %The number $d_B(A)$, if it exists, is \corr{the asymptotic density of the set $A$ in the set $B$.
  The asymptotic density $d_B(A)$ can also be interpreted as an asymptotic probability of
finding a $\lambda$-term  from the class $A$ among all $\lambda$-terms
from $B$.
  \item $d_B$ is finitely additive: if $A_1$ and $A_2$ are
disjoint classes of $\lambda$-terms  such that $d_B (A_1)$ and $d_B
(A_2)$ exist then $d_B (A_1 \cup A_2)$ also exists and
 $d_B (A_1 \cup A_2) = d_B (A_1) + d_B (A_2) .$
  \item It is straightforward to observe that for any infinite $B$,
    meeting the condition of definition \ref{density}, and finite set
$A$ the density $d_B (A)$ exists and is $0$. Dually for co-finite
sets $A$ the density $d_B (A) = 1$.
\item The density $d_B$ is not countably additive,
so in general the formula

\begin{equation} \label{countably_additive}
d_B \left( \bigcup_{i=0}^\infty A_i \right) =  \sum_{i=0}^\infty
d_B( A_i) \nonumber
\end{equation}

is not true for all classes of \corr{pairwise disjoint} sets
$\set{A_i}_{i\in\mathbb{N}}$. A % good
counterexample for the
equation
%\ref{countably_additive}
is to take $B = \Lambda$ and $A_i$ the singleton containing the $i$-th lambda
term from our language under any natural \corr{enumeration} of terms. On
the left hand side of the equation %\ref{countably_additive}
we get $d_{\Lambda}(\Lambda)$ which is $1$ but on
right hand side
%of (\ref{countably_additive})
$d_{\Lambda}( A_i) = 0$ for all $i\in \mathbb{N}$
and so the sum is $0$.
\item Let $P$ be a property of
\correc{closed} $\lambda$-terms.  If $d_{\Lambda}(\{t \in \Lambda \ | \ t \text{
satisfies } P \})=\alpha$, we  say that the density of terms
satisfying $P$ is $\alpha$. By analogy to research on graphs
and trees, whenever we say that ``a random term satisfies $P$'' we
mean that ``the density of terms satisfying $P$ is $1$''.

\end{enumerate}
\section{Proofs using calculus}\label{calculus}

In this section we state a few theorems which provide bounds for $L_n$ (the number of \correc{closed} $\lambda$-terms of size $n$).
We also find a lower bound for the unary height in a random term.

\subsection{Lower bound for $L_n$}\label{lower_bound}

The estimation for $L_n$ \correct{which} we provide is \corr{rather imprecise but sufficient for our purpose}.

\begin{thm}\label{lower}
For any $\varepsilon \in (0,4)$  we have
$$L_n \gtrsim \left(\frac{(4-\varepsilon)n}{\ln(n)}\right)^{n-\frac{n}{\ln(n)}} .$$
\end{thm}

\proof
Let $LB(n,k)$ denote the number of closed $\lambda$-terms of size $n$ with $k$ head lambdas and no other $\lambda$ below. Since the lower part of the term is a binary tree with $n-k$ inner nodes and each leaf can be bound by $k$ lambdas, we have $LB(n,k) = C(n-k) k^{n-k+1}$. Clearly, $L_n \geq LB(n,k)$ for all $k=1,\ldots,n$. Let $k= \left\lceil \frac{n}{\ln(n)} \right\rceil$. Then we get:
\begin{align*}
L_n &\geq C\left( n- \left\lceil \frac{n}{\ln(n)} \right\rceil \right) \left( \left\lceil \frac{n}{\ln(n)} \right\rceil \right) ^{n- \left\lceil \frac{n}{\ln(n)} \right\rceil \correc{+}1} & \\
&\sim \frac{4^{n- \left\lceil\frac{n}{\ln(n)}\right\rceil } }{ {\left( n- \left\lceil\frac{n}{\ln(n)}\right\rceil \right)^{3/2}\sqrt{\pi}}} \left( \left\lceil \frac{n}{\ln(n)} \right\rceil \right) ^{n- \left\lceil \frac{n}{\ln(n)} \right\rceil \correc{+}1} & \text{by Proposition \ref{catalan}} \\
&\gtrsim \left( \frac{4 n}{\ln(n)} \right) ^{n-\frac{n}{\ln(n)}} \frac{1}{p(n)} & \text{for some positive polynomial $p$} \\
&\gtrsim \left(\frac{(4-\varepsilon)n}{\ln(n)}\right)^{n-\frac{n}{\ln(n)}} & \text{since $\left(\frac{4}{4-\varepsilon}\right)^{n - \frac{n}{\ln n}} \gtrsim p(n)$}.
\end{align*}
\qed

\subsection{Number of lambdas in a term}

In this part we focus on the number of unary and binary nodes in random $\lambda$-terms. We need the following lemma:

\begin{lem}\label{incr_decr}
For all sufficiently large $n$, the function $f(p) = p^{n-p+1}$ is
\begin{enumerate}[\em(i)]
\item decreasing on $[\frac{3 n}{\ln(n)}, \correc{+\infty})$,
\item increasing on $(0, \frac{n}{3 \ln(n)}]$.
\end{enumerate}
\end{lem}

\proof
Let us start \corr{by} computing the derivative of the function $f$ on
$(0, +\infty)$:
$$f'(p) = \left( p^{n-p+1} \right)' = \left( e^{(n-p+1)\ln(p)} \right)' = e^{(n-p+1)\ln(p)} \left( \frac{n-p+1}{p} -\ln(p) \right) .$$

\begin{enumerate}[(i)]
\item We want to show that $f'(p) < 0$ for any $p \in \left[ \frac{3 n}{\ln(n)}, \correc{+\infty} \right)$. This is equivalent to the following inequality:
$n+1 < p(\ln(p) +1)$. The expression on the right reaches the minimum in the considered interval at $p = \frac{3 n}{\ln(n)}$, thus it is sufficient to prove that
$$n+1 < \frac{3 n}{\ln(n)} \left(\ln\left(\frac{3 n}{\ln(n)}\right) +1 \right) .$$
But the right expression is equal to
\begin{align*}
\frac{3 n}{\ln(n)}(\ln(n) - \ln(\ln(n)) + \ln 3 +1 ) \cr
&\hspace{-5em} = 2n +
\frac{n}{\ln(n)}(\ln(n) - 3\ln(\ln(n)) + 3\ln 3 +3 )\cr 
&\hspace{-5em}>n+1,
\end{align*}
  which finishes the proof. The last inequality is obvious for  sufficiently large $n$.
%The last step uses $- 3 \ln(\ln(n)) \geq - \frac{3 \ln(n)}{2} + 3 - 3\ln(2)$ obtained using the fact that the logarithm function is below is tangent at $2$: $\ln(x) \leq \frac{x}{2} - 1 + \ln(2)$ for $x > 0$.

\item We want to show that $f'(p) > 0$ for any $p \in \left( 0, \frac{n}{3 \ln(n)} \right]$. This is equivalent to the following inequality:
$n+1 > p(\ln(p) +1)$. The expression on the right reaches the maximum in the considered interval at $p = \frac{n}{3 \ln(n)}$, thus it is sufficient to prove that
$$n+1 > \frac{n}{3 \ln(n)} \left(\ln\left(\frac{n}{3 \ln(n)}\right) +1 \right) .$$
But the right expression is equal to
%$
%\begin{array}{ll}
%\frac{n}{3 \ln(n)} \left(\ln(n) - \ln(\ln(n)) - \ln 3 +1 \right)
%&= \correc{\frac{n}{3} - \frac{n}{3 \ln(n)} \left( \ln(\ln(n)) + \ln 3 - 1 \right)} \cr
%&\hspace{-3cm} = \correc{n - \frac{n}{3 \ln(n)} \left(2\ln(n) + \ln(\ln(n)) + \ln 3 - 1 \right)} \cr
%&\hspace{-3cm} < n+1,
%\end{array}$
\begin{align*}
\frac{n}{3 \ln(n)} \left(\ln(n) - \ln(\ln(n)) - \ln 3 +1 \right)\cr
&\hspace{-5em} = \correc{\frac{n}{3} - \frac{n}{3 \ln(n)} \left( \ln(\ln(n)) + \ln 3 - 1 \right)} \cr
&\hspace{-5em} = \correc{n - \frac{n}{3 \ln(n)} \left(2\ln(n) + \ln(\ln(n)) + \ln 3 - 1 \right)} \cr
&\hspace{-5em} < n+1,
\end{align*}

\noindent which finishes the proof. The last inequality is obvious for  sufficiently large $n$.
%The last step using $\ln(3) > 1$ and $\ln(\ln(2)) + 2 \ln(2) > 0$ for $n = 2$. 
\qed%\smallskip
\end{enumerate}\nobreak
\noindent The next theorem shows that the typical proportion of unary nodes \correct{to}  binary ones in $\lambda$-terms is far from the typical proportion in ordinary unary-binary trees, in which case it tends \corr{to} a positive constant.\newpage

\begin{nota}
\label{nota:A}
Let ${\mathcal A}$ denote the class of closed terms $t \in {\mathcal A}$
that \correct{satisfies} all the following conditions:

\begin{enumerate}[(i)]
\item the number of lambdas in $t$ is at most $\frac{3 \size (t)}{\ln(\size (t))}$,
\item the number of lambdas in $t$ is at least $\frac{\size (t)}{3 \ln(\size (t))}$,
\item the unary height of $t$ is at least $\frac{\size (t)}{3 \ln(\size (t))}$.
\end{enumerate}
\end{nota}

\begin{thm}\label{lambda-bound}\label{depth-bound}
The density of ${\mathcal A}$ in $\Lambda$ is $1$.
\end{thm}

\proof
Let us consider terms of size $n$ with exactly $p$ lambdas. Such terms
have exactly $n-p+1$ leaves and each of them can be bound by at most
$p$ lambdas. Since the number of unary-binary trees of size $n$ and
with $n-p+1$ leaves is equal to $M(n,n-p+1)$ (see \ref{schroder}), we
obtain the following upper bound for the number of considered terms: $p^{n-p+1} M(n,n-p+1)$.

Now, we show that each of properties (i)--(iii) characterizing the class ${\mathcal A}$ is valid for random terms. Obviously, property (iii) implies property (ii), but our proof of (iii) uses (ii) as intermediate result so we make it explicit.

\begin{enumerate}[(i)]
\item Let $P_n$ denote the number of closed terms of size $n$ containing more
  than $\frac{3 n}{\ln(n)}$ lambdas. We have $P_n \leq \sum_{p
    \geq \frac{3 n}{\ln(n)}} p^{n-p+1} M(n,n-p+1)$.\\
By Lemma \ref{incr_decr} the function $p \mapsto p^{n-p+1}$ is decreasing in the interval $\left[\frac{3 n}{\ln(n)},n \right]$. Thus,
$$P_n \leq \sum_{p \geq \frac{3 n}{\ln(n)}} M(n,n-p+1) \left(\frac{3 n}{\ln(n)}\right)^{n+1-\frac{3 n}{\ln(n)}} \leq M(n) \left(\frac{3 n}{\ln(n)}\right)^{n+1-\frac{3 n}{\ln(n)}}.$$
By the lower bound for $L_n$ from \ref{lower_bound} and the computations above, we get
$$\frac{P_n}{L_n} \lesssim \frac{M(n)\left(\frac{3 n}{\ln(n)}\right)^{n+1-\frac{3 n}{\ln(n)}}} {\left(\frac{(4-\varepsilon)n}{\ln(n)}\right)^{n-\frac{n}{\ln(n)}}} .$$
To get the result it remains to show that for some $\varepsilon \in (0,4)$ this expression tends to 0.
By Proposition \ref{motzkin}, $M(n) \sim \left(\frac{1}{3-2\sqrt{2}}\right)^n \frac{1}{\sqrt{\pi} n^\frac{3}{2}}$. Using this equivalence, we deduce that there is some positive constant $\gamma$ such that we have:
\begin{align*}
\frac{P_n}{L_n} &\lesssim \gamma \frac{\left(\frac{1}{3-2\sqrt{2}}\right)^n \left(\frac{3 n}{\ln(n)}\right)^{n+1-\frac{3 n}{\ln(n)}}}{n^{\frac{3}{2}}\left(\frac{(4-\varepsilon) n}{\ln(n)}\right)^{n-\frac{n}{\ln(n)}}} & \\
&\lesssim \frac{\left(\frac{1}{3-2\sqrt{2}}\right)^n \left(\frac{3
      n}{\ln(n)}\right)^{n-\frac{3
      n}{\ln(n)}}}{\left(\frac{(4-\varepsilon) n}{\ln(n)}\right)
  ^{n-\frac{n}{\ln(n)}}} &  \hspace{-2em}  \text{since $\frac{3\gamma n}{\ln (n)} \lesssim n^{\frac32}$}\\
&= \left(\frac{3}{(4-\varepsilon)(3-2\sqrt{2})}\right)^n \left(\frac{3 n}{\ln(n)}\right)^{\frac{-3n}{\ln(n)}}\left(\frac{(4-\varepsilon) n}{\ln(n)}\right)^{\frac{n}{\ln(n)}} & \\
&= \left(\frac{3}{(4-\varepsilon)(3-2\sqrt{2})}\right)^n \left( \frac{3^{-3}(4-\varepsilon)\ln ^{2} (n)}{n^2} \right) ^{\frac{n}{\ln (n)}} &
\end{align*}
Notice that for any $\alpha$, $\left( n^{2-\alpha} \right) ^{n / \ln(n)} =
e^{\ln (n) {(2-\alpha) \frac {n}{ \ln (n)}}} = e^{(2-\alpha)n}$.
Thus, we obtain
$$\frac{P_n}{L_n} \lesssim \left(\frac{3}{(4-\varepsilon) (3-2\sqrt{2}) e^{2-\alpha}}\right)^n \left(3^{-3} (4-\varepsilon) \frac{\ln^{2}(n)}{n^{\alpha}}\right)^{\frac{n}{\ln(n)}}.$$
Let $\alpha$ and $\varepsilon$ be positive and small enough so that $3 < (4-\varepsilon) (3-2\sqrt{2}) e^{2-\alpha} $. Then the whole expression tends to $0$ as $n$ tends to infinity\correc{, which finishes the proof}.

\item Let $R_n$ denote the number of terms of size $n$ containing less
  than $\frac{n}{3 \ln(n)}$ lambdas. We have $R_n \leq \sum_{p \leq
    \frac{n}{3 \ln(n)}} p^{n-p+1} M(n,n-p+1) $.\\
By Lemma \ref{incr_decr} the function $p \mapsto p^{n-p+1}$ is
increasing in the interval $\left[0,\frac{n}{3 \ln(n)} \right]$. Thus,
\begin{align*}
R_n &\leq \sum_{p \leq \frac{n}{3 \ln(n)}} M(n,n-p+1) 
\left(\frac{n}{3 \ln(n)}\right)^{n+1-\frac{n}{3 \ln(n)}} \cr &\leq
M(n) \left(\frac{n}{3 \ln(n)}\right)^{n+1-\frac{n}{3 \ln(n)}}.
\end{align*}
By the lower bound for $L_n$ from Theorem \ref{lower} and the computations above, we get
\begin{align*}
\frac{R_n}{L_n} & \lesssim \frac{M(n)\left(\frac{n}{3\ln(n)}\right)^{n+1-\frac{n}{3\ln(n)}}} {\left(\frac{(4-\varepsilon)n}{\ln(n)}\right)^{n-\frac{n}{\ln(n)}}} & \\
& \lesssim \gamma \frac{\left(\frac{1}{3-2\sqrt{2}}\right)^n \left(\frac{n}{3\ln(n)}\right)^{n+1-\frac{n}{3\ln(n)}}} {n^{\frac{3}{2}}\left(\frac{(4-\varepsilon) n}{\ln(n)}\right)^{n-\frac{n}{\ln(n)}}} & \text{for some } \gamma > 0\\
&\lesssim \frac{\left(\frac{1}{3-2\sqrt{2}}\right)^n
  \left(\frac{n}{3\ln(n)}\right)^{n-\frac{n}{3\ln(n)}}}{\left(\frac{(4-\varepsilon)
      n}{\ln(n)}\right) ^{n-\frac{n}{\ln(n)}}} &\text{since $\frac{\gamma n}{3\ln (n)} \lesssim n^{\frac32}$}\\
&= \left(\frac{1}{3(4-\varepsilon)(3-2\sqrt{2})}\right)^n \left( \frac{3(4-\varepsilon)^3 n^2}{(\ln(n))^2} \right) ^{\frac{n}{3\ln (n)}} & \\
& = \left(\frac{e^{2/3}}{3(4-\varepsilon)(3-2\sqrt{2})}\right)^n \left( \frac{3(4-\varepsilon)^3 }{(\ln(n))^2} \right) ^{\frac{n}{3\ln (n)}} & \text{since } n^{\frac{2n}{3\ln(n)}} = e^{\frac{2}{3} n}.
\end{align*}
For $\varepsilon >0$ small enough the whole expression tends to $0$, \correct{which} finishes the proof.

\item Let $S_n$ be the number of closed terms of size $n$ with more than $\frac{n}{3 \ln(n)}$ lambdas and with the unary height less than $\frac{n}{3\ln(n)}$.
Such a term has at most $n - \frac{n}{3 \ln(n)} + 1$ leaves and each of them can be bound by one of at most $\frac{n}{3 \ln(n)}$ lambdas. Therefore, we have
$$ S_n \leq M(n) \left(\frac{n}{3\ln(n)}\right)^{n - \frac{n}{3 \ln(n)} + 1}$$
Dividing it by the lower bound for $L_n$ and performing exactly the same calculations as in the proof of (ii), we obtain the desired result.  \qed
\end{enumerate}

\subsection{Upper bound for $L_n$}\label{upper_bound}

Now we are ready to provide an upper bound for $L_n$. Once again, this estimation is very rough, however, it turns out to be sufficient for our main goal.

\begin{lem}\label{subexp}
Let $\alpha(n)$ be either ${n\mapsto \bigl\lceil\frac{3n}{\ln (n)}\bigr\rceil}$ or ${n\mapsto \bigl\lfloor\frac{3n}{\ln (n)}\bigr\rfloor}$.
Then the function $n \mapsto {3n \choose {\alpha(n)}}$ is subexponential. 
\end{lem}

\newcommand\myexpr{\frac{3n}{\ln (n)}}
\newcommand\roundup{{\left\lceil\myexpr\right\rceil}}

\proof
%For $n$ large enough, $\roundup$ is less than $\frac{n}{2}$ so that ${{2n+1 \choose \alpha(n)}\leq{2n+1 \choose\roundup}}$. Hence, we have:
%  \begin{align*}
%    {2n+1 \choose {\alpha(n)}} &= \frac{2n+1}{2n+1-\alpha(n)} \frac{(2n)!}{\left(2n-\alpha(n)\right)! \left(\alpha(n)\right)!} \\
%   &\lesssim \frac{(2n)!}{\left(2n-\roundup\right)! \left(\roundup\right)!} 
%  \end{align*}
Using the Stirling formula
\[n! \sim \sqrt{2 \pi n} \left( \frac{n}{e} \right)^n\]
we obtain, for some polynomial function $\gamma(n)$, the asymptotic majoration:
\begin{align*}
{3n \choose {\alpha(n)}}
& \lesssim \gamma(n) \frac{(3n)^{3n}}{\left(3n-\roundup\right) ^{3n-\roundup} \left(\roundup\right) ^{\roundup}} \\
& \lesssim\gamma(n) E(n)
\end{align*}
where $E(n)$ can be written
\[E(n) = \frac{3^{3n}}{\left(3-\frac{\roundup}{n}\right)^{3n-\roundup}\left(\frac{\roundup}{n}\right)^{\roundup}}\]
Let us compute the logarithm of $E(n)$:
\begin{align*}
  \ln\bigl(E(n)\bigr) &= 3n \ln(3) - \left( 3n-\roundup\right) \ln \left(
    3-\frac{\roundup}{n} \right) - \roundup{}\ln \left(
    \frac{\roundup}{n} \right)\\
  &\leq 3n \ln(3) - \left( 3n-\myexpr-1\right) \ln \left(
    3-\frac{3}{\ln(n)}-\frac{1}{n} \right) - \left(\myexpr{} +1\right)\ln \left(
    \frac{3}{\ln(n)}\right)\\
\end{align*}
After some simplifications we obtain that $\ln\bigl(E(n)\bigr) \lesssim 3n \frac{\ln\ln(n)}{\ln(n)} + o \left( n \frac{\ln(\ln(n))}{\ln(n)} \right)$. Since the polynomial function $\gamma(n)$ belongs to ${o\left(e^{\alpha n\frac{\ln(\ln(n))}{\ln(n)}}\right)}$ for any positive $\alpha$, we finally deduce that:
\[{3n \choose \alpha(n)} \lesssim e^{\delta n \frac{\ln\ln(n)}{\ln(n)}} \quad \text{ for some } \delta > 0 .\eqno{\qEd} \]

\begin{thm}\label{upper}
For any $\varepsilon > 0$  we have $$L_n \lesssim \left(\frac{(12+\varepsilon) n}{\ln(n)}\right)^{n-\frac{n}{3\ln(n)}}$$
\end{thm}

\proof
Let $T_n$ be the number of terms of size $n$ with less than $\frac{3n}{\ln(n)}$ and more than $\frac{n}{3\ln(n)}$ lambdas. According to Theorem \ref{lambda-bound} we have $L_n \sim T_n$. In $\lambda$-terms enumerated by $T_n$ the number of binary nodes is at most $n - \frac{n}{3 \ln(n)}$ and the number of leaves is at most greater by one. We compute the upper bound for $T_n$ in the following way:
\begin{enumerate}[$\bullet$]
\item first, we consider binary trees built on at most $n - \left\lfloor\frac{n}{3 \ln(n)}\right\rfloor$ binary nodes --- their number does not exceed Catalan number $C \left( n- \left\lfloor \frac{n}{3\ln(n)} \right\rfloor +1 \right)$ (the $+1$ in the argument is obtained through Proposition~\ref{catalan} because we sum $C(i)$ over all possible $i$ up to $n - \left\lfloor\frac{n}{3 \ln(n)}\right\rfloor$),
\item then, we insert in such trees at most $\frac{3n}{\ln(n)}$ (the maximum number of lambdas) unary nodes --- this can be done in less than ${3n \choose \left\lceil\frac{3n}{\ln(n)}\right\rceil}$ ways ($3n-\left\lceil\frac{3n}{\ln(n)}\right\rceil$ is an upper bound for the number of possible places for insertions into a binary tree of size $n-\frac{n}{3\ln(n)}+1$),
\item finally, we have at most $n+1-\frac{n}{3\ln(n)}$ leaves in such trees and each of them can by bound by at most $\frac{3n}{\ln(n)}$ lambdas --- thus the number of possible ways of binding is not greater than $\left(\frac{3n}{\ln(n)}\right)^{n+1-\frac{n}{3\ln(n)}}$.
\end{enumerate}
Thus, we get
$$ T_n \lesssim C\left( n- \left\lfloor\frac{n}{3\ln(n)}\right\rfloor +1 \right) {3n \choose \roundup} \left(\frac{3n}{\ln(n)}\right)^{n+1-\frac{n}{3\ln(n)}} .$$
Using the asymptotic expansion of Catalan numbers (Proposition \ref{catalan}), we obtain
\begin{align*}
T_n &\lesssim {3n \choose {\roundup}} \frac{4^{n - \lfloor \frac{n}{3\ln(n) }\rfloor+1}}{\sqrt{\pi} \left(n - \frac{n}{3\ln(n)}+1 \right)^{3/2} } \left(\frac{3n}{\ln(n)}\right)^{n+1-\frac{n}{3\ln(n)}} \\
& \lesssim {3n \choose {\roundup}} \left(\frac{12 n}{\ln(n)}\right)^{n-\frac{n}{3\ln(n)}} \\
& \lesssim \left(\frac{(12+\varepsilon) n}{\ln(n)}\right)^{n-\frac{n}{3\ln(n)}},
\end{align*}
for any $\varepsilon>0$. The last line follows from the fact that ${2n+1 \choose {\roundup}}$ is subexponential (by Lemma \ref{subexp}).
\qed

\textbf{Remark.} The ratio between the upper and lower bounds obtained for $L_n$ is exponential, but $L_n$ is super-exponential itself.

% \subsection{Comparison between the lower and the upper bounds}

% The ratio between \correc{the upper and lower bounds is less than}
% $$\left( \frac{12 + \varepsilon}{4 - \varepsilon} \right) ^{n - \frac{n}{3\ln (n)}},$$
% \correc{but still} exponential. But, since $L_n$ is super-exponential itself, our estimations are not too bad.

% The following corollary shows that we know the first two terms of the asymptotic expansion of $\ln(L_n)$, but we do not know the linear factor yet.

% \begin{corollary}
% For any $\varepsilon \in (0,4) $ and for $n$  large enough
% \begin{align*}
% n\ln(n) - n\ln(\ln(n)) + n\left(\ln(4-\varepsilon) - 1\right) \cr &\hspace{-8.5em} \leq \ln(L_n) \cr &
% \hspace{-8.5em} \leq  n\ln(n) - n\ln(\ln(n)) + n\left(\ln(12+\varepsilon) -
% \frac{1}{3}\right) + \frac{n\ln(\ln(n))}{3\ln(n)}.
% \end{align*}
% \end{corollary} 

\section{Proofs using coding}\label{coding}

In this section we prove theorems about random $\lambda$-terms using the following scheme. First, we consider a set
$\Lambda_n(\mathcal{P})$ of terms of size $n$  satisfying some
property $\mathcal{P}$. Next, we define an injective and
size-preserving function $\varphi^{\mathcal P}_n \colon
\Lambda_n(\mathcal{P}) \to \Lambda_n$ (called a {\it coding}) such
that its image has density $0$ among all closed lambda
terms. \corr{This is sufficient to prove that this property is not
  satisfied by random terms.}

%More precisely, we do not directly use this scheme with $\mathcal{P}$ the property of not being strongly normalizable. Instead
We consider successive sets of terms $X_1,\ldots, X_k$ with $X_{i+1}\subseteq X_i$ and we prove:
\begin{enumerate}[(1)]
\item $X_1$ has density $1$ (Theorem \ref{depth-bound});
\item $X_{i+1}$ has density $1$ because $X_i\setminus X_{i+1}$ has density $0$ (successive theorems of this section).
\end{enumerate}
By choice of $X_k$, we finally get that SN terms have density $1$.
Below, these sets $X_1,X_2,\ldots$ are denoted ${\mathcal A,B,}\ldots$ and depend on some parameters (integers or functions).

%Some proofs do not use coding: one may consider they use the identity as coding.
Some %other
proofs need the following lemma:

\begin{lem}\label{convergence}
  Let $A_n$ be a sequence of \correc{non empty finite sets} of terms and $B_n$
  be subsets of $A_n$.
  Let
  $(A_{n,i})_{i \in I_n }$ be a partition of $A_n$ and let
  $B_{n,i} =A_{n,i} \cap B_n$. Let $a_n$ (resp.\ $b_n$, $a_{n,i}$, $b_{n,i}$) be the
  cardinality of $A_n$ (resp.\ $B_n$, $A_{n,i}$, $B_{n,i}$).  Assume
  $\frac{b_{n,i}}{a_{n,i}}$ tends to 0 uniformly in $i$ as $n$ tends
  to infinity, formally:
  \[\forall \varepsilon >0, \exists N, \forall n \geq N, \forall i\in I_n:\frac{b_{n,i}}{a_{n,i}} \leq \varepsilon.\]
  Then $\frac{b_n}{a_n}$ tends to 0 as $n$ tends to infinity.
\end{lem}

\proof
Let $\varepsilon>0$. Let $N$ be the corresponding integer guaranteed by the uniform convergence and let $n$ be any integer with $n\geq N$. We have:
\begin{align*}
  \frac{b_n}{a_n} = \frac{\sum_{i\in I_n} b_{n,i}}{a_n}
  = \sum_{i\in I_n} \frac{b_{n,i}}{a_{n,i}}\frac{a_{n,i}}{a_n}
  \leq \sum_{i\in I_n} \varepsilon \frac{a_{n,i}}{a_n}
  \correc{=} \varepsilon.
\end{align*}
We have shown ${\displaystyle\lim_{n\rightarrow\infty}\frac{b_n}{a_n}= 0}$.
\qed

\subsection{The number of lambdas in head position}

We start with showing that a random term starts with a long chain of lambdas. In the next theorem and until the end of the paper, we denote by $g$ a lower bound on the length of this chain (as a function of the size of the term). Theorem~\ref{th:starting_lambdas} below shows that any $g \in o\big( \sqrt{n/\ln(n)} \big)$ is an admissible lower bound. However, the reader can think of $g$ as the function ${n\mapsto\ln(n)^2+3}$ since the main theorem (Theorem~\ref{main}) and all intermediate results can be proved using this particular choice of $g$ (see Proposition~\ref{prop:unsafedensity}).

\begin{nota}
  Let $g \colon {\mathbb N} \to {\mathbb N}$
  %be a function such that   $g \in o\big( \sqrt{n/\ln(n)} \big)$. Let us
  We define ${\mathcal B}^g$
  as the class of terms $t$ such that
\begin{enumerate}[1.]
\item $t \in {\mathcal A}$ \correc{(see Notation~\ref{nota:A})},
\item $t$ has at least $g(\size (t))$ head lambdas.
\end{enumerate}
Additionally, we denote by $\overline{\mathcal{B}^g}={\mathcal A} \setminus {\mathcal B}^g$ the complement of the set ${\mathcal B}^g$ in ${\mathcal A}$ and by $\overline{\mathcal{B}^g_n}$ the set of terms from $\overline{\mathcal{B}^g}$ of size $n$.
\end{nota}

\begin{thm}\label{th:starting_lambdas}
Let $g \colon {\mathbb N} \to {\mathbb N}$ be a function such that
$g \in o\big( \sqrt{n/\ln(n)} \big)$. The density of ${\mathcal
B}^g$ in $\Lambda$ is $1$.
\end{thm}
\newcommand\bbgn{\overline{\mathcal{B}^g_n}}
\newcommand\bbgnpart[1]{\bbgn(#1)}
\newcommand\tvec{\vec{t}}
\newcommand\bbgntl{\bbgnpart{\vec{t},\ell}}

\proof
  Let us fix $g \in o\big( \sqrt{n/\ln(n)} \big)$. Our aim is to
  construct a family of injective and size-preserving functions
  (codings) $\varphi_n^{\mathcal{B}} \colon \bbgn
  \to \Lambda_n$ such that the fraction ${\#\varphi_n^{\mathcal{B}}
    \left( \bbgn \right) } / L_n$ tends to $0$ as
  $n$ tends to infinity.

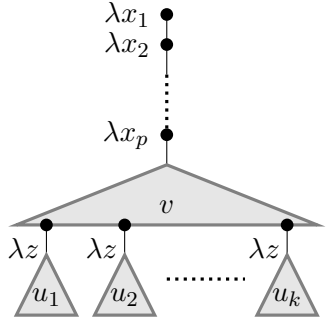
\begin{figure}[ht]
  \centering
  \begin{tikzpicture}[scale=.8]
    \lmbda{(0,4)}{(0,3.5)}{$\lambda x_1$}{east}
    \lmbda{(0,3.5)}{(0,3)}{$\lambda x_2$}{east} \draw[dotted, very
    thick] (0,3) -- (0,2); \lmbda{(0,2)}{(0,1.5)}{$\lambda x_p$}{east}
    \subterm{(0,1)}{2.5}{.5}{$v$} \lmbda{(-2,.5)}{(-2,0)}
    {$\lambda
      z$}{north east} \subtermspec{(-2,-.5)}{.5}{.5}{$u_1$}{(0,-.2)}
   \lmbda{(-.7,.5)}{(-.7,0)}{$\lambda z$}{north east}
    \subtermspec{(-.7,-.5)}{.5}{.5}{$u_2$}{(0,-.2)}
    \lmbda{(2,.5)}{(2,0)}{$\lambda
      z$}{north east} \subtermspec{(2,-.5)}{.5}{.5}{$u_k$}{(0,-.2)}
    \draw[dotted, very thick] (0,-.4) -- (1.3,-.4);
  \end{tikzpicture}

  \caption{\rm A term from $\overline{\mathcal{B}^g_n}(\protect\vec{t},\ell)$ where $\protect\vec{t} =(\lambda z.u_1,\ldots,\lambda z.u_k)$
  }
  \label{fig:headlambdaterm}
\end{figure}

  Let $n_0 > 1$ be such that $g(n) <
  \frac{n}{3\ln(n)}$ for all $n \geq n_0$. Such $n_0$ exists
  %by   definition of ${\mathcal A}$ and
  because $g \in o\big(\sqrt{n/\ln(n)} \big)$. In the rest of the proof we always assume that $n
  \geq n_0$.

We define a partition %$\left(\bbgntl\right)_{\tvec,\ell}$ on
of  $\bbgn$ as follows \corr{(see
Figure~\ref{fig:headlambdaterm})}. Let $\tvec$ be a non-empty sequence of
  (not necessarily closed) terms such that each of the elements of
  $\vec{t}$ starts with a $\lambda$. Let $\ell\geq 1$ be an
  integer such that $0 \leq n - \ell - \size (\vec{t}\,) \leq
  g(n)$, where $\size (\vec{t}\,)$ denotes the sum of sizes of
  its components. We define $\bbgntl$ as the set of terms of the form:
  \[\lambda x_1\ldots\lambda x_p .v[t_1,\ldots,t_k]\]
  where $v$ is a purely applicative context with $k$ holes,
  $\tvec=(t_1,\ldots,t_k)$ and ${p=n-\ell - \size(\tvec\,)}$. Therefore,
$\ell$ is the size of the applicative context $v$ (where the hole are
counted with size $0$ like variables).

%   $\overline{\mathcal{B}^g}(n,\ell,\vec{t}\,)$ the set of
%   terms from $\overline{\mathcal{B}^g_n}$ which can be decomposed as
%   follows:
% \begin{enumerate}[(i)]
% \item at the their top there are $p = n - \ell - \size (\vec{t}\,)$ head lambdas,
% \item below them there is a purely applicative prefix $v$ of size $\ell$ whose leaves are either
% variables introduced by lambdas from (i) or $|\vec{t}|$
% empty places where  $|\vec{t}|$   is the length of the
% sequence $\vec{t}$
% \item terms from $\vec{t}$ are plugged into empty places of the prefix $v$.
% \end{enumerate}
\noindent   First, it is clear that nonempty sets $\bbgntl$ form a partition of
  $\overline{{\mathcal B}^g_n}$: they are pairwise disjoint by
  definition and every $u\in\bbgn$ belongs to $\mathcal A$ so it
  contains some $\lambda$ not in the chain of head lambdas (because $p
  \leq g(n) <
  \frac{n}{3\ln(n)}$),
  therefore it belongs to some $\bbgntl$ for some non-empty $\tvec$
  and some $\ell\geq 1$.

  Terms from $\bbgntl$ differ only by applicative contexts, so
  the cardinality of $\bbgntl$ is less than the number of all binary
  trees of size $\ell$ in which each leaf is either labeled with a
  variable (for which we have at most ${g(n)-1}$ possibilities) or is
  an empty place where some sub-term can be plugged. Thus, we have for
  all $n \geq n_0$:
  \[\#\bbgntl\leq P(n, \ell):=C(\ell)(g(n))^{\ell+1}.\]
Let ${t\in \bbgntl}$
%we define $\varphi_{n,\vec{t},\ell}(t)$ as follows: Let
and
$\vec{t}=(t_1,\ldots,t_k)$ for some $k\geq 1$ and $v$
be the purely applicative context in the decomposition of $t$.
We can write $t_i=\lambda z.u_i$. Consider the term
\[t'=\lambda z \lambda x_1\ldots\lambda x_p.(u_1 \ (u_2 \ (\ldots(u_{k-1} \ u_k)\ldots) ))\]
which is of size
\[n-\ell = n\ \underbrace{-\ \ell}_{v\text{ \tiny removed}}\
\underbrace{-\ k}_{\text{\tiny \corr{head lambdas from $t_i$ removed}}}\ \underbrace{+\ 1}_{\text{\tiny head $\lambda z$}}\ \underbrace{+\ k-1}_{\text{\tiny applicative nodes}}.\]
\correc{We rename bound variables,
%we consider
so that a variable distinct from $z$ in $t$ is renamed to  $x_k$  where $k$ is number of lambdas from the root to the lambda binding that variable (inclusive).  Let $V_n$
be the set of variables $\{x_1,\ldots,x_{\left\lceil\frac{n}{3\ln(n)}\right\rceil}\}$.}
Let $\lambda y.s$ denote the term
rooted at the leftmost deepest $\lambda$ of term $t'$.

Since the unary height of
$t$ is the same as \corr{that} of $t'$, \correc{and since $t\in\mathcal{A}$, % by Theorem \ref{depth-bound}(iii)
all the variables in $V_n\,$ are bound on the
path from the root to $\lambda y.s$ (in the worst case, $y$ is
$x_{\left\lceil\frac{n}{3\ln(n)}\right\rceil}$ and must also be
counted on the path).}

Let $U_{n,l}$ be the set of purely applicative (therefore not closed) terms
of size $\ell-1$ whose variables are chosen from $V_n$.There are at
least
$$Q(n,\ell)=C(\ell-1) \left( {\frac{n}{3\ln(n)}} \right) ^\ell$$
elements in $U_{n,l}$.

Let $\psi(n,\ell)=\frac{P(n,\ell)}{Q(n,\ell)}$. By the assumption \correct{about}
$g$, there is a function $\varepsilon$ such that
$\displaystyle\lim_{n\rightarrow\infty}\varepsilon(n)=0$ and
$P(n,\ell)\leq
C(\ell)\left(\sqrt{\frac{n}{\ln(n)}}\varepsilon(n)\right)^{\ell+1}\!\!.$ Therefore,
we have \[\psi(n,\ell)\leq\frac{C(\ell)}{
  3\,C(\ell-1)}\left(\frac{n}{\ln(n)}\right)^{\frac{1-\ell}{2}}(3\,\varepsilon
(n))^{\ell+1}.\]
For $\ell \geq 1$, $\left(\frac{n}{\ln(n)}\right)^{\frac{1-\ell}{2}}$
is decreasing in $\ell$ and since \correc{$\frac{C(\ell)}{C(\ell-1)} = \frac{2(2\ell-1)}{\ell+1}$}, it follows that $\psi(\ell,n)$ tends to $0$ uniformly in
$\ell$.

\correc{From this, for $n$ large enough, we get $P(n,\ell) < Q(n,\ell)$
  (uniform convergence of $\psi$ is needed only later) and there exists an
injective function $h_{n,\ell}$ which assigns  an
element from $U_{n,l}$ to any purely applicative
context using variables in $\{x_1,\ldots,x_p\}$ (i.e.\ applicative context $v$ used in the decomposition of a term in
$\overline{\mathcal{B}^g_n}(\protect\vec{t},\ell)$).}

For any $u\in U_{n,l}$, let $\rho(t',u)$ be the term obtained by substituting
the subterm $\lambda y.s$ in $t'$ with $\lambda y.(u \ s)$.

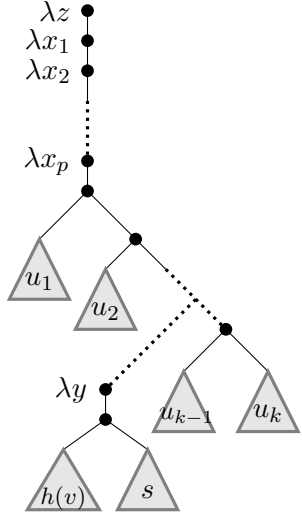
\begin{figure}[ht]
  \centering
  \begin{tikzpicture}[scale=.8]
    \lmbda{(-1.3,5.3)}{(-1.3,4.8)}{$\lambda z$}{east}
    \lmbda{(-1.3,4.8)}{(-1.3,4.3)}{$\lambda x_1$}{east}
    \lmbda{(-1.3,4.3)}{(-1.3,3.8)}{$\lambda x_2$}{east} \draw[dotted,
    very thick] (-1.3,3.8) -- (-1.3,2.8);
    \lmbda{(-1.3,2.8)}{(-1.3,2.3)}{$\lambda x_p$}{east}
    \app{(-1.3,2.3)}{(-2.1,1.5)}{(-.5,1.5)}
    \subtermspec{(-2.1,1)}{.5}{.5}{$u_1$}{(0,-.2)}
\app{(-.5,1.5)}{(-1,1)}{(0,1)}
    \subtermspec{(-1,.5)}{.5}{.5}{$u_2$}{(0,-.2)} \draw[dotted,very thick](0,1) -- (1,0); \draw[dotted, very thick] (.5,.5) -- (-1,-1);
    \lmbda{(-1,-1)}{(-1,-1.5)}{$\lambda y$}{east}
    \app{(-1,-1.5)}{(-1.7,-2)}{(-.3,-2)}
    \subtermspec{(-1.7,-2.5)}{.6}{.5}{\footnotesize $h(v)$}{(0,-.3)}
    \subtermspec{(-.3,-2.5)}{.5}{.5}{$s$}{(0,-.2)}
    \app{(1,0)}{(.3,-.7)}{(1.7,-.7)}
    \subtermspec{(1.7,-1.2)}{.5}{.5}{$u_k$}{(0,-.2)}
   \subtermspec{(.3,-1.2)}{.5}{.5}{{\small $u_{k-1}$}}{(.05,-.25)}
  \end{tikzpicture}
  \caption{The term $\varphi_{n}^{\mathcal{B}}(t)$ from Theorem \ref{th:starting_lambdas}}
\label{fig:headlambdacode}
\end{figure}

\correc{Let $\varphi_{n,\vec{t},\ell}(t)=\rho(t',h_{n,\ell}(v))$} (see
Figure~\ref{fig:headlambdacode}). It is easy to check that the
size of $\varphi_{n,\vec{t},\ell}$ is $n$ and that, by
the injectivity of $h_{n,\ell}$, $\varphi_{n,\vec{t},\ell}$ is
injective, too.

Let $\varphi_{n}^{\mathcal{B}} = \bigcup_{\ell, \vec{t}}
\varphi_{n,\vec{t},\ell}$. The function $\varphi_{n}^{\mathcal{B}}$ is
an injection because codomains of the
$\varphi_{n,\vec{t},\ell}$ are all disjoint by %the
construction. Since the sets $\bbgntl$ form a partition of
$\overline{\mathcal{B}^g_n}$, \correc{by means of \corr{Lemma} \ref{convergence}}, it is
enough to show that
$\frac{P(n,\ell)}{Q(n,\ell)}$ tends uniformly in $l$ to $0$ as $n$
tends to infinity, which was done above.
\qed

\subsection{Head lambdas bind ``many'' occurrences}

Now we are ready to present some theorems showing that in a random term head lambdas are used, i.e.\ they really bind some variables.
The first result shows that in a random term many of head lambdas are binding.

\begin{nota}
Let $g \colon {\mathbb N} \to {\mathbb N}$ be a function such that
$g \in o\big( \sqrt{n/\ln(n)} \big)$. By ${\mathcal D}^g$ we
denote the class of terms such that $t \in {\mathcal D}^g$ iff
\begin{enumerate}[1.]
\item $t \in {\mathcal B}^{g+1}$, where $g+1$ is the function $n \mapsto g(n)+1$,
\item each of first $g(\size (t))$ head lambdas in $t$ is binding.
\end{enumerate}
Additionally, we denote by $\overline{\mathcal{D}^g_n} = {\mathcal B}^{g+1} \setminus {\mathcal D}^g$ the \correc{complement} of the class ${\mathcal D}^g$ in ${\mathcal B}^{g+1}$ and by $\overline{\mathcal{D}^g_n}$ the set of terms from $\overline{\mathcal{D}^g}$ of size $n$.

\end{nota}

\begin{thm}\label{binding1}
Let $g \colon {\mathbb N} \to {\mathbb N}$ be a function such that
$g \in o\big( \sqrt{n/\ln (n)} \big)$. The density of ${\mathcal
D}^g$ in $\Lambda$ is $1$.
\end{thm}

\proof
Let us fix $g \in o\big( \sqrt{n/\ln (n)} \big)$. We construct a
family of codings $\varphi_{n}^{\mathcal{D}} \colon
\overline{\mathcal{D}^g_n} \to \Lambda_n$ such that their images
are negligible in $\Lambda _n$, i.e.  the fraction
$\sharp \varphi_n^{\mathcal D} (\overline{\mathcal{D}^g_n}) / L_n$ tends to $0$ as
$n$ tends to infinity.

Let $t=\lambda x_1 \ldots x_{g(n)+1}.u$ be a term from $\overline{\mathcal{D}^g_n}$ and let $i \leq g(n)$ be the smallest integer such that the $i$-th head lambda in $t$ does not bind any variable. Take
$$\varphi_n^{\mathcal{D}}(t) := \lambda x_1 \ldots x_{i-1} x_{i+1}.\big(x_{i+1} \ (\lambda x_{i+2}\ldots x_{g(n)+1}.u)\big).$$
The size of $\varphi_n^{\mathcal{D}}(t)$ is $n$. Terms from the set $\varphi_n^{\mathcal{D}}(\overline{\mathcal{D}_n^g})$ have less than $g(n)\correc{+1}$ head lambdas. By Theorem \ref{th:starting_lambdas}, the density of such terms in $\Lambda $ is zero. Since the function $\varphi_{n}^{\mathcal{D}}$ is injective, the density of $\overline{\mathcal{D}^g}$ is also zero.
\qed

\begin{nota}
Let $g,h \colon {\mathbb N} \to {\mathbb N}$ be functions such
that $g \in o\big( \sqrt{n/\ln (n)} \big)$, $g(n) \geq 3$ for all
$n$ and $h \in o\left(\log
_3\left(\frac{n}{\ln(n)}\right)\right)$. By ${\mathcal E}^{g,h}$
we denote the class of closed terms such that $t \in {\mathcal
E}^{g,h}$ iff
\begin{enumerate}[1.]
\item $t \in {\mathcal D}^g$,
\item the total number of occurrences of variables bound by \corr{the} first three lambdas in $t$ is greater than $h(\size (t))$.
\end{enumerate}
Additionally, we denote by $\overline{\mathcal{E}^{g,h}} = {\mathcal D}^{g} \setminus {\mathcal E}^{g,h}$ the complement of the class ${\mathcal E}^{g,h}$ in ${\mathcal D}^{g}$ and by $\overline{\mathcal{E}^{g,h}_n}$ the set of terms from $\overline{\mathcal{E}^{g,h}}$ of size $n$.
\end{nota}

\begin{thm}\label{binding2}
Let $g,h \colon {\mathbb N} \to {\mathbb N}$ be functions such
that $g \in o\big( \sqrt{n/\ln (n)} \big)$, $g(n) \geq 3$ for all
$n$ and $h \in o\left(\log
_3\left(\frac{n}{\ln(n)}\right)\right)$. The density of ${\mathcal
E}^{g,h}$ in $\Lambda$ is $1$.
\end{thm}

\proof
Let $g$ and $h$ be functions as in the assumptions of the theorem. We construct a family of codings $\varphi^{\mathcal{E}}_{n} \colon \overline{\mathcal{E}^{g,h}_n} \to \Lambda_n$ such that their images are negligible in $\Lambda _n$ as $n$ tends to infinity.

Let us define an equivalence relation $\sim_n$ on the set of terms of size $n$ in the following way: $u \sim_n v$ iff $u$ and $v$ are equal after substituting all occurrences of variables bound by first three lambdas by the variable bound by the first $\lambda$. Let us denote by $[u]$ the equivalence
class of $u$.

Let $t=\lambda x_1 \lambda x_2 \lambda x_3 .u$ be a term from $\overline{\mathcal{E}^{g,h}_{n}}$. There are at most $3^{h(n)}$ elements in the class $[t]$.

Let $\psi(t)=\lambda x y.u[x_1:=y, x_2:=y, x_3:=y]$. The size of $\psi(t)$ is $n-1$. Let $\lambda a.v$ \corr{be} the subterm of $\psi(t)$ such that $\lambda a$ is the leftmost deepest $\lambda$ in $\psi(t)$.
Denote by $V(t)$ the set of variables introduced by lambdas occurring in $\psi(t)$ on the path from $\lambda a$ to $\lambda y$.
Note that the variable $x$ \corr{occurs \correct{neither} in $\psi(t)$ \correct{nor}} in $V(t)$.

By Theorem \ref{depth-bound}(iii), there are at least
$\frac{n}{3\ln(n)}-2$ such lambdas.
%Since $2\leq \frac{n}{6\ln(n)}$, there are at least $\frac{n}{6\ln(n)}$ elements in $V(t)$.
As $h \in o\left(\log
_3\left(\frac{n}{\ln(n)}\right)\right)$, we have
$$\lim_{n\to \infty}\frac{3^{h(n)}}{\big(\frac{n}{3\ln(n)}-2\big)}=0.$$
Thus, we can find for each class $[t]$ an injective function $f_{[t]}$ from $[t]$ into the set $V(t)$.

We define $\varphi^{\mathcal E}_{n}(t)$ as the term obtained from $\psi(t)$ by replacing the subterm $\lambda a.v$ with $\lambda a. (w \ v)$, where $w=f_{[t]}(t)$.

All terms from the image $\varphi^{\mathcal E}_n(\overline{\mathcal{E}^{g,h}_{n}})$ start with a $\lambda$ that binds no variable. By Theorem \ref{binding1} we know that such terms are negligible in $\Lambda_n$. Since $\varphi^{\mathcal E}_n$ is injective, the density of $\overline{\mathcal{E}^{g,h}}$ is zero, as well.
\qed

\begin{nota}
Let $k$ and $\ell$ be natural numbers. Let $g \colon {\mathbb N} \to
{\mathbb N}$ be functions such that $g \in o\big( \sqrt{n/\ln (n)}
\big)$, $g(n) \geq 3$ for all $n$, $\lim_{n\rightarrow\infty}g(n)=\infty$, and let $h(n)=\left\lfloor
\sqrt{\log _3 \left(\frac{n}{\ln(n)}\right)} \right\rfloor$.
Notice that $h \in o\left(\log
_3\left(\frac{n}{\ln(n)}\right)\right)$. By ${\mathcal G}^{g,
k,\ell}$ we denote the class of closed terms such that $t \in
{\mathcal G}^{g,k,\ell}$ iff
\begin{enumerate}[1.]
\item $t \in {\mathcal E}^{g,h}$,
\item each of first $k$ lambdas in $t$ binds more than $\ell$ variables.
\end{enumerate}
Additionally, we denote by $\overline{\mathcal{G}^{g,k,\ell}} = {\mathcal E}^{g,h} \setminus {\mathcal G}^{g,k,\ell}$ the complement of the class ${\mathcal G}^{g,k,\ell}$ in ${\mathcal E}^{g,h}$ and by $\overline{\mathcal{G}^{g,k,\ell}_n}$ the set of terms from $\overline{\mathcal{G}^{g,k,\ell}}$ of size $n$.
\end{nota}

\begin{thm}\label{binding3}
Let $k$ and $\ell$ be integers. Let $g \colon {\mathbb N} \to
{\mathbb N}$ be a function such that $g \in o\big( \sqrt{n/\ln
(n)} \big)$, $\lim_{n\rightarrow\infty}g(n)=\infty$, and $g(n) \geq 3$ for all $n$. The density of
${\mathcal G}^{g,k,\ell}$ in $\Lambda$ is $1$.
\end{thm}

\proof
Let $g$ be a function as in the assumptions of the theorem and let us fix integers $k$ and $\ell$. Without loss of generality we can assume that $k \geq 3$. By Theorem \ref{binding2}, the total number of occurrences of variables bound by first $k$ lambdas in terms from $\overline{\mathcal{G}^{g,k,\ell}_n}$ is greater than $h(n) = \left\lfloor \sqrt{\log _3 \left(\frac{n}{\ln(n)}\right)} \right\rfloor$.

For $m \geq h(n)$ let us denote by $\mathcal{E}^{g,h}_n (m,k)$ the set of terms from ${\mathcal{E}^{g,h}_n}$ with exactly $m$ leaves bound by \correc{the} first $k$ lambdas and let $\overline{\mathcal{G}^{g,k,\ell}_{n}}(m) = \overline{\mathcal{G}^{g,k,\ell}_n} \cap {\mathcal{E}^{g,h}_n}(m,k)$. By definition, terms from $\overline{\mathcal{G}^{g,k,\ell}_{n}}(m)$ have exactly $m$ leaves bound by \corr{the} first $k$ lambdas and at least one of \correct{these} lambdas binds at most $\ell$ variables.

Consider the equivalence relation $\sim_n$ on
${\mathcal{E}^{g,h}_n(m,k)}$ defined analogously \correc{to} the
relation \corr{with the same notation within} the proof of Theorem \ref{binding2}, but with respect to the first $k$ (instead of three) head lambdas. \correc{Denote by $[t]$ the equivalence class of $t$ for \correct{this} relation.}

Let $t \in {\mathcal{E}^{g,h}_n}(m,k)$. By hypothesis on $g$ and for large enough $n$, each of \correc{the} first $k$ head lambdas of $t$ are binding. \correc{Of the $m$ leaves bound by these lambdas, give the $k$ leftmost leaves distinct labels}. For each of $m-k$ other leaves we have $k$ possibilities. Thus, we know that the cardinality of $[t]$ is greater than $k^{m-k}$.

Now, let us estimate the upper bound for the cardinality of $[t]\cap
\overline{\mathcal{G}^{g,k,\ell}_{n}}(m)$. In such terms there exists
at least one lambda among first $k$ \correct{lambdas} which binds
$\ell'$ leaves with $1 \leq \ell' \leq \ell$ (we can \correct{choose} them out of $m$ ones) and
the other leaves (their number is equal to $m-\ell' \leq m - 1$) can be bound
by $k-1$ lambdas. Thus, we obtain the upper bound equal to
$\sum_{1 \leq \ell'\leq  \ell} k {m
  \choose  \ell'} (k-1)^{m- \ell'} \leq k m^ \ell (k-1)^{m-1}$.
This holds because $\sum_{1 \leq  \ell'\leq  \ell} {m \choose  \ell'} \leq m^ \ell$
which can be proved by induction over $ \ell$ when $m \geq 2$ (here $m
\geq k \geq 3$).

Therefore, the quotient of the two cardinalities is less than
$$\frac{k m^{\ell} (k-1)^{m-1}}{k^{m-k}} = k^k m^{\ell}
\left( \frac{k-1}{k} \right) ^{m-1} \hspace{-1.5em}  \text{ for all } m \geq h(n).$$
As $n$ tends to infinity, the above quotient tends to $0$ uniformly in
$m$. To establish this, we define $f(x) = x^\ell R^{\frac{x}{2}}$ with $R
  = \frac{k-1}{k} < 1$. Thus we have
$$ k^k m^{\ell}
\left( \frac{k-1}{k} \right) ^{m-1} \leq k^k f(m) R  ^{\frac{m}{2}-1}$$
Then, $f'(x) = x^{ \ell-1}  R^{\frac{x}{2}} ( l + x \frac{\ln(R)}{2} )$ and
  we see that $f(x)$ reaches a maximum on ${\mathbb R}_+$ for $x = A =
  -\frac{2 \ell}{\ln(R)}$ (which is a positive constant because $R =  \frac{k-1}{k} < 1$),
which gives:
$$ k^k m^{\ell}
\left( \frac{k-1}{k} \right) ^{m-1} \leq k^k f(A) R^{\frac{m}{2}-1}
\leq k^k f(A) R^{\frac{h(n)}{2}-1} $$
For $t \in \Lambda_n$ and $m \geq h(n)$ \correc{the} sets $[t] \cap
\overline{\mathcal{G}^{g,k,\ell}_{n}}(m)$ form a partition of
$\overline{\mathcal{G}^{g,k,\ell}_{n}}$. Now Lemma \ref{convergence} finishes the proof.
\qed

As a simple corollary of the above theorem, we obtain the following result:

\begin{nota}
Let $k$ and $\ell$ be positive integers. Let $g \colon {\mathbb N} \to {\mathbb N}$ be a
 function such that $g \in o\big( \sqrt{n/\ln (n)} \big)$, $\lim_{n\rightarrow\infty}g(n)=\infty$, and $g(n) \geq 3$ for all $n$.
 By ${\mathcal H}^{g, k,\ell}$ we denote the class of terms such that $t \in {\mathcal H}^{g, k,\ell}$
 iff
\begin{enumerate}[1.]
\item $t \in {\mathcal G}^{g, k,\ell}$,
\item there are no two consecutive non-binding lambdas in $t$.
\end{enumerate}
Additionally, we denote by $\overline{\mathcal{H}^{g, k,\ell}} = {\mathcal G}^{g,k,l} \setminus {\mathcal H}^{g,k,\ell}$ the complement of the class ${\mathcal H}^{g,k,l}$ in ${\mathcal G}^{g,k,\ell}$ and by $\overline{\mathcal{H}^{g,k,\ell}_n}$ the set of terms from $\overline{\mathcal{H}^{g,k,\ell}}$ of size $n$.
\end{nota}

\begin{lem}\label{lem:twononbind}
Let $k$ and $\ell$ be positive integers. Let $g \colon {\mathbb N}
\to {\mathbb N}$ be a function such that $g \in o\big( \sqrt{n/\ln
(n)} \big)$, $\lim_{n\rightarrow\infty}g(n)=\infty$, and $g(n) \geq 3$ for all $n$. The density of
${\mathcal H}^{g, k,\ell}$ in $\Lambda$ is $1$.
\end{lem}

\proof
We define a family of injective and size-preserving functions $\varphi_{n}^{\mathcal{H}}$ from $\overline{\mathcal{H}^{g, k,\ell}_n}$ into the set of terms whose leading $\lambda$ binds only one \corr{variable occurrence}.

Let $t$ be a term from ${\overline{\mathcal H}_n^{g, k,\ell}}$. Let
$t_1$ be a subterm rooted at a highest \corr{leftmost} occurrence of two non-binding lambdas, $t_1=\lambda x.\lambda y.u$. We replace this subterm by the application $(z \ u)$, where $z$ is a fresh variable. We obtain the term $t'$ of size $n-1$ and, finally, we define $\varphi_n^{\mathcal H}(t) = \lambda z.t'$. The result follows from Theorem~\ref{binding3}.
\qed

\subsection{A random term avoids any fixed closed term}

%\begin{nota}
%Let $\overline{\Lambda}_{j}^i$ denote the set of (not necessarily closed) terms of size $j$ and with $i$ occurrences of free variables.
%\end{nota}

\begin{nota}
Let $j$ be a positive integer and $k(j) = \sum_{i\leq j}L_i$ (let
us recall that $L_i$ denotes the number of closed terms of size
$i$). Let $g \colon {\mathbb N} \to {\mathbb N}$ be a function
such that $g \in o\big( \sqrt{n/\ln (n)} \big)$, $g(n) \geq 3$ for
all $n$ and $\lim_{n \to \infty}g(n)=\infty$. By ${\mathcal
I}^{g,j}$ we denote the class of closed terms such that $t \in
{\mathcal I}^{g,j}$ iff
\begin{enumerate}[1.]
\item $t \in {\mathcal H}^{g,k(j),k(j)}$,
\item $t$ does not contain any term from $\bigcup_{i \leq j} \Lambda_{\corr{i}}$ as a subterm.
\end{enumerate}
Additionally, we denote by $\overline{\mathcal{I}^{g,j}} = {\mathcal H}^{g,k(j),k(j)} \setminus {\mathcal I}^{g,j}$ the complement of the class ${\mathcal I}^{g,j}$ in ${\mathcal H}^{g,k(j),k(j)}$ and by $\overline{\mathcal{I}^{g,j}_n}$ the set of terms from $\overline{\mathcal{I}^{g,j}}$ of size $n$.
\end{nota}

\begin{thm}\label{avoidaux}
Let $j$ be a positive integer and let $g \colon {\mathbb N} \to
{\mathbb N}$ be a function such that $g \in o\big( \sqrt{n/\ln
(n)} \big)$, $g(n) \geq 3$ for all $n$ and $\lim_{n \to
\infty}g(n) = \infty$. The density of ${\mathcal I}^{g,j}$ in
$\Lambda$ is $1$.
\end{thm}

\proof
Let us fix a positive integer $j$ and a function $g$ as in the assumptions of the theorem. We construct a family of codings $\varphi_n^{\mathcal{I}} \colon \overline{\mathcal{I}^{g,j}_n} \to \Lambda_n$ such that their images are negligible in $\Lambda _n$.

There are $k(j)=\sum_{i \leq j} L_i$ elements in $\bigcup_{i \leq j}
\Lambda_i$. Thus, there is a bijection $f$ from $\bigcup_{i \leq j}
\Lambda_i$ to $\{ 1, \ldots, k(j)\}$.
% Let $t_0 \in \bigcup_{i \leq j} \Lambda_i$ and let $f(t_0)=m$, where $m \leq k(j)$.

Let $n$ be an integer satisfying $g(n) > k(j)$ and $n > k(j)+j$. Let $t \in \overline{\mathcal{I}^{g,j}_n}$. By hypothesis the term $t$ belongs to ${\mathcal B}^{g+1}$, so it has more than $k(j)$ head lambdas since $k(j) < g(n)$ (see Figure~\ref{fig:avoidterm}).

\begin{figure}[ht]
\centering
\begin{minipage}{.4\linewidth}
  \begin{tikzpicture}[scale=.8]
      \lmbda{(0,4)}{(0,3.5)}{$\lambda x_1$}{east}
      \lmbda{(0,3.5)}{(0,3)}{$\lambda x_2$}{east}
      \draw[dotted, very thick] (0,3) -- (0,2);
      \lmbda{(0,2)}{(0,1.5)}{$\lambda x_{k(j)}$}{east}
      \subterm{(0,1)}{.5}{.5}{}
      \subtermspec{(0,0)}{.5}{.5}{$u$}{(0,-.2)}
  \end{tikzpicture}
\end{minipage}
\begin{minipage}{.4\linewidth}
  \begin{tikzpicture}[scale=.8]
      \lmbda{(0,4)}{(0,3.5)}{$\lambda x_1$}{east}
      \lmbda{(0,3.5)}{(0,3)}{$\lambda x_2$}{east}
      \draw[dotted,very thick] (0,3) -- (0,2);
      \lmbda{(0,2)}{(0,1.5)}{$\lambda x_{m-1}$}{east}
      \lmbda{(0,1.5)}{(0,1)}{$\lambda x$}{east}
      \lmbda{(0,1)}{(0,.5)}{$\lambda x_{m}$}{east}
      \draw[dotted,very thick] (0,.5) -- (0,-.5);
      \lmbda{(0,-.5)}{(0,-1)}{$\lambda x_{k(j)}$}{east}
      \subterm{(0,-1.5)}{.5}{.5}{}
      %\app{(0,-2)}{(-.7,-2.5)}{(.7,-2.5)}
      \subterm{(-0,-2.5)}{.5}{.5}{$v$}
  \end{tikzpicture}
\end{minipage}
\caption{\rm Terms $t \in \overline{\mathcal{I}^{g,j}_n}$ and $\varphi_n^{\mathcal{I}}(t)$ from Theorem \ref{avoidaux}}
\label{fig:avoidterm}
\end{figure}
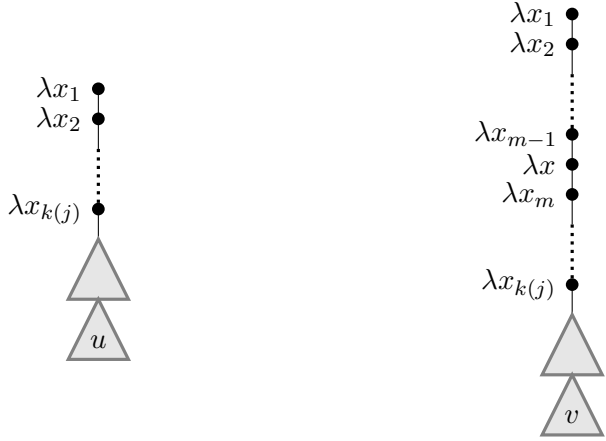

In term $t$, consider the smallest $m$ such that $f(u) = m$ for some closed $u$ occurring in $t$
%occurs in $t$
(there is at least one such $m$ because $t \in \overline{\mathcal{I}^{g,j}_n}$.
Let us consider the term $s$ which is obtained from the term $t$ by adding an additional unary node (labeled with $\lambda x$) at depth $m$.  Let us define $\varphi_n^{\mathcal{I}}(t)$ obtained by replacing the leftmost deepest \correc{occurrence of} subterm $u$ in $s$ by the term $v=(x\ (x\ (\ldots (x\ x)\ldots)))$ of size $\correc{i}-1$ \correc{where $i$ is the size of $u$} (see Figure~\ref{fig:avoidterm}). Thus, the size of $\varphi_n^{\mathcal{I}}(t)$ is equal to $n$.

By Theorem \ref{binding3}, each of the first $k(j)$ head lambdas in \corr{a} term
from ${\mathcal H}^{g,k(j),k(j)}$ of size $n$ binds more than $k(j)$
variables. \correc{Therefore, among the first $k(j)$ head lambdas of
  $\varphi_n^{\mathcal{I}}(t)$, only the $m$-th $\lambda$ binds less than $k(j)$ variables
  (recall that $u$ is closed which means that the number of
  variables bound by $\lambda x_i$ for $1 \leq i \leq k(j)$ is the same in $t$ and $\varphi_n^{\mathcal{I}}(t)$). Since $f(u)=m$ and $f$ is injective, the function $\varphi_n^{\mathcal{I}}$ is injective.} Terms from the image $\varphi_n^{\mathcal{I}}(\overline{\mathcal{I}^{g,j}_n})$ \correc{are not in ${\mathcal H}^{g,k(j),k(j)}$} since the $m$-th $\lambda$ binds only $\correc{i\leq}j\leq k(j)$ variables. Thus, those terms are negligible among all terms of size $n$.
\qed

%Considering two special cases of the above theorem we obtain:

%\begin{corollary}\label{avoid}
%Let $t_0$ be a (not necessarily closed) term. If $t_0$ is closed or if there are at least two lambdas in
%$t_0$, the density of terms containing $t_0$ as a subterm is $0$.
%\end{corollary}

\subsection{The $\lambda$-width of a term}

Let us recall that $\lambda$-width of a term is the maximum number of incomparable binding lambdas in the term. In the following proposition we show that $\lambda$-width of typical $\lambda$-terms is small.

\begin{nota}
Let $g \colon {\mathbb N} \to {\mathbb N}$ be a function such
 that $g \in o\big( \sqrt{n/\ln (n)} \big)$, $g(n) \geq 3$ for all $n$ and $\lim_{n \to \infty}g(n)=\infty$.
  By ${\mathcal J}^{g}$ we denote the class of closed terms such that $t \in {\mathcal J}^{g}$
   iff
\begin{enumerate}[1.]
\item $t \in {\mathcal G}^{g,1,4}$
%(we really need that the first $\lambda$  bind at least 4 times and for later proofs that the first   $g(\size(t)))$ are binding),
\item $\lambda$-width of $t$ is at most % than
  $2$.
\end{enumerate}
Additionally, we denote by $\overline{\mathcal{J}^{g}} = {\mathcal G}^{g,1,4} \setminus {\mathcal J}^{g}$ the complement of the class ${\mathcal J}^{g}$ in ${\mathcal G}^{g,1,4}$ and by $\overline{\mathcal{J}^{g}_n}$ the set of terms from $\overline{\mathcal{J}^{g}}$ of size $n$.
\end{nota}

\begin{thm}\label{th:widthtwo}
Let $g \colon {\mathbb N}
\to {\mathbb N}$ be a function such that $g \in o\big( \sqrt{n/\ln
(n)} \big)$, $g(n) \geq 3$ for all $n$ and $\lim_{n \to
\infty}g(n)=\infty$. The density of ${\mathcal J}^{g}$ in
$\Lambda$ is $1$.
\end{thm}

\proof
Let us fix a function $g$ as in the assumptions
of the theorem. We construct a family of codings
$\varphi_n^{\mathcal{J}} \colon \overline{\mathcal{J}^{g}_n} \to
\Lambda_n$ such that their images are negligible in $\Lambda _n$. Let
$t$ be an element of $\overline{\mathcal{J}^{g}_n}$, therefore the
$\lambda$-width of $t$ is at least $3$. Let us denote by $\lambda x$,
$\lambda y$ and $\lambda z$ the first three highest \corr{leftmost} pairwise incomparable binding lambdas (appearing in this order from left to right in $t$).

    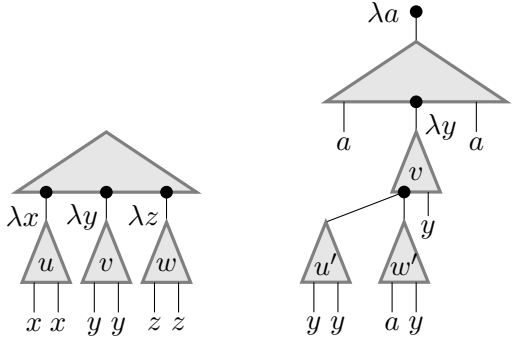
\begin{figure}[ht]
    \centering
    \begin{tikzpicture}[scale=.8]
      \subterm{(0,3)}{1.5}{.5}{}
      \lmbda{(-1,2.5)}{(-1,2)}{$\lambda x$}{north east}
      \subtermspec{(-1,1.5)}{.4}{.5}{$u$}{(0,-.2)}
      \vari{(-1.2,1)}{(-1.2,.5)}{$x$}
      \vari{(-.8,1)}{(-.8,.5)}{$x$}
      \lmbda{(0,2.5)}{(0,2)}{$\lambda y$}{north east}
      \subtermspec{(0,1.5)}{.4}{.5}{$v$}{(0,-.2)}
      \vari{(-.2,1)}{(-.2,.5)}{$y$}
      \vari{(.2,1)}{(.2,.5)}{$y$}
      \lmbda{(1,2.5)}{(1,2)}{$\lambda z$}{north east}
      \subtermspec{(1,1.5)}{.4}{.5}{$w$}{(0,-.2)}
      \vari{(.8,1)}{(.8,.5)}{$z$}
      \vari{(1.2,1)}{(1.2,.5)}{$z$}
    \end{tikzpicture}\quad\quad\quad
    \begin{tikzpicture}[scale=.8]
      \lmbda{(0,4)}{(0,3.5)}{$\lambda a$}{east}
      \subterm{(0,3)}{1.5}{.5}{}
      \vari{(-1.2,2.5)}{(-1.2,2)}{$a$}
      \vari{(1,2.5)}{(1,2)}{$a$}
      \begin{scope}[xshift=-1cm]
        \lmbda{(1,2.5)}{(1,2)}{$\lambda y$}{north west}
        \subtermspec{(1,1.5)}{.4}{.5}{$v$}{(0,-.2)}
        \vari{(1.2,1)}{(1.2,.6)}{$y$}
\app{(.8,1)}{(-.5,.5)}{(.8,.5)}
        \subtermspec{(-.5,0)}{.4}{.5}{$u'$}{(0,-.2)}
        \vari{(-.7,-.5)}{(-.7,-1)}{$y$}
        \vari{(-.3,-.5)}{(-.3,-1)}{$y$}
        \subtermspec{(.8,0)}{.4}{.5}{$w'$}{(0,-.2)}
        \vari{(.6,-.5)}{(.6,-1)}{$a$} \vari{(1,-.5)}{(1,-1)}{$y$}
      \end{scope}
    \end{tikzpicture}
    \caption{\rm The terms $t$ and $\varphi_n^{\mathcal{J}}(t)$ from Theorem \ref{th:widthtwo}}
    \label{fig:lambdawidth}
\end{figure}

Let $\lambda x. u$, $\lambda y. v$ and $\lambda z. w$ be subterms rooted at those lambdas (see Figure~\ref{fig:lambdawidth}). Let $u'=u[x:=y]$, let $a$ be a new variable, \corr{and} let $w'$ be the term obtained from $w$ by replacing the leftmost occurrence of $z$ with $a$ and the others (possibly none) with $y$. Let $\varphi_n^{\mathcal{J}}(t)$ be the term obtained from $t$ by adding $\lambda a$ at the root, substituting both subterms $\lambda x.u$ and $\lambda z.w$ with $a$ and replacing the leftmost occurrence of $y$ in $v$ with term $(u' \ w')$. We have $\size(\varphi_n^{\mathcal{J}}(t))= \size(t)$. Also note that since we \correct{choose} the highest three incomparable binding lambdas no variable becomes free in the constructed term.

%The injectivity of $\varphi_n^{\mathcal{J}}$ comes from the fact that both $\lambda y$ and the subterm $(u' \ w')$ of $\varphi_n^{\mathcal{J}}(t)$ are uniquely %identifiable (see Figure~\ref{fig:lambdawidth}):

We can reconstruct the term $t$ from $\varphi_n^{\mathcal J}(t)$ by indicating places for $\lambda y$ and the subterm $(u' \ w')$:

\begin{enumerate}[$\bullet$]
\item Let $\nu_l$ (resp.\ $\nu_r$) be the deepest node above the two leftmost (resp.\ rightmost) occurrences of $a$. Remark that since there \corr{are} exactly $3$ occurrences of $a$, one of these two nodes is above the other. Let $\nu$ be the deepest one. $\lambda y$ is the first binding $\lambda$ on the path from the    node $v$ to the middle occurrence of $a$;
\item then, the application node $(u' \ w')$ is the deepest node above the middle occurrence of $a$ and all the occurrences of $y$ on the left of this middle occurrence of $a$.\smallskip
\end{enumerate}

\noindent Since the image of $\varphi_n^{\mathcal J}$ contains only terms starting with a $\lambda$ which binds only $3$ occurrences of the corresponding variable, by Theorem \ref{binding3}, the density of $\varphi_n^{\mathcal J}(\overline{\mathcal{J}^{g}_n})$ is equal to zero. The injectivity of $\varphi_n^{\mathcal J}$ finishes the proof.
\qed

\subsection{The density of strongly normalizable terms}

From Theorem~\ref{th:widthtwo} (using $g(n)=\ln(n)^2+3$ for instance) we know that almost all terms are of width at most $2$. In Section \ref{lambda} we introduced the notion of 'safe'
terms of width $2$ which implies strong normalization (Proposition~\ref{width2sn}).

Now we prove that the set of unsafe terms of width $2$ has density $0$.

\begin{nota}
Let $g \colon {\mathbb N}
\to {\mathbb N}$ be a function such that $g \in o\big( \sqrt{n/\ln
(n)} \big)$, $g(n) \geq 3$ for all $n$ and $\lim_{n \to
\infty}g(n)=\infty$. By ${\mathcal K}^{g}$ we denote the class
of closed terms such that $t \in {\mathcal K}^{g}$ iff
\begin{enumerate}[1.]
\item $t \in {\mathcal J}^{g}$,
\item $t$ is safe.
\end{enumerate}
Additionally, we denote by $\overline{\mathcal{K}^{g}} = {\mathcal J}^{g} \setminus {\mathcal K}^{g}$ the complement of the class ${\mathcal K}^{g}$ in ${\mathcal J}^{g}$ and by $\overline{\mathcal{K}^{g}_n}$ the set of terms from $\overline{\mathcal{K}^{g}}$ of size $n$. Note that terms from $\overline{\mathcal{K}^{g}}$ are of $\lambda$-width at most $2$ and are unsafe, therefore they are of width exactly $2$ (because terms of width $1$ are safe by definition).
\end{nota}

\begin{prop}\label{prop:unsafedensity}
Let $g \colon {\mathbb N} \to {\mathbb N}$ be the function defined by $g(n)=\ln(n)^2+3$.
  The density of ${\mathcal K}^{g}$ in $\Lambda$ is $1$.
\end{prop}

\proof
The root of the minimal subterm of width $2$ of a term is called the branching node and is always binary.
We show that the density of $\overline{\mathcal{K}^{g}}$ in $\Lambda$ is $0$. Let us divide the set $\overline{\mathcal{K}^{g}_n}$ into two disjoint subsets:

\bigskip

\noindent $\overline{\mathcal{K}^{g,1}_{n}}$: the set of terms of size $n$ such that \correc{neither of} the lengths of paths from the branching node to the two highest incomparable binding lambdas \correc{is} greater than $\ln(n)$,

\bigskip

\noindent $\overline{\mathcal{K}^{g,2}_{n}}$: the set of remaining terms.

\bigskip

We can construct a family of codings from the set $\overline{\mathcal{K}^{g,1}_{n}}$ into $\Lambda_n$ in the following two steps:

\correct{
\begin{enumerate}[(1)]
\item Remove \corr{the} two highest pairwise incomparable binding lambdas and put one lambda, binding their variables, at the root of the whole term. The size of the obtained term is smaller by $1$ and the branching node is uniquely determined.
 \item Insert one non-binding lambda among the head lambdas of the term. By choice of $g$ and by definition of ${\mathcal{K}^g}$, terms from $\overline{\mathcal{K}^{g,1}_{n}}$ have more than $\ln(n)^2$ head lambdas. Therefore we can encode the lengths of the paths from the branching node to the two highest binding lambdas as the position of this new lambda. By Theorem \ref{binding2} the image of such \corr{a} transformation has density $0$.\smallskip
\end{enumerate}
}

\noindent For the set $\overline{\mathcal{K}^{g,2}_{n}}$ \corr{we do not construct an
  injection, but a relation that associates to terms in
  $\overline{\mathcal{K}^{g,2}_{n}}$ disjoint set of terms of
  cardinals greater than  $\ln(n)/2$.} This is enough to show that $\overline{\mathcal{K}^{g,2}_{n}}$ has density $0$. Precisely, we proceed as follows:
 \correct{
 \begin{enumerate}[(1)]
 \item Choose the \corr{leftmost} path among the one or two paths longer than
   $\ln(n)$  (without loss of generality we
   can assume it is the left path, the case of the right one is
   analogous).
% and connects the branching node and one of two highest binding
% lambdas.
 \corr{Consider the binding lambda at the end of this path and let}
 $t_0$ be the subterm rooted at this lambda.
Let $t_1,\dots, t_k$ be the right subtrees rooted at the
 binary nodes on the path between the branching node and $t_0$
(the path goes always to the left since the term is unsafe).
 By Lemma~\ref{lem:twononbind} at
 least half of the nodes on this path are binary (since there are no
 two consecutive non-binding lambdas in the tree). This means that $k
 \geq \ln(n)/2$.
  Moreover,
 the terms $t_1,\ldots, t_k$ contain no binding lambda otherwise, the lambda
 width of the term would be greater than $2$.

\item  Choose some leaf $x$ belonging to some subtree \corr{among}
  $t_1,\dots, t_k$ and exchange it with the subterm
  $t_0$. Independently of \corr{the} choice of the leaf, the encoding
  can be reversed since:
  \begin{enumerate}\item the position of $t_0$ in the encoded term
  is uniquely identifiable
    as the highest binding lambda of the innocuous subtree below the
    branching node (\corr{the innocuousness identifies the modified
      branch});
    \item the position of the variable $x$ in the encoded term is
      identifiable as the leftmost leaf of the subtree rooted at the
      branching node of the resulting term which is still of width 2
      (in the case of the right branch, it is the leftmost leaf of the
      right sub-term of the branching node).
  \end{enumerate}
The encoding preserves size and the number of
  possibilities for the choice of a leaf $x$ is
  the number of leafs of $t_1,\dots,t_k$, which is greater than
 $\ln(n)/2$. Therefore, terms from
  $\overline{\mathcal{K}^{g,2}_{n}}$ are negligible in $\Lambda_n$
  as $n$ tends to infinity.\qed

\end{enumerate}
}

\begin{main_theorem}\label{main}
  The set of strongly normalizable terms has density $1$.
\end{main_theorem}

\proof
  Proposition~\ref{prop:unsafedensity} shows the existence of a set of safe terms that has density $1$.
  Proposition~\ref{width2sn} shows that they are all strongly normalizable.
\qed
\section{Combinatory logic}\label{cl}

In this section we show that our main result about strong normalization of random $\lambda$-terms does not \corr{hold} in the world of random combinatory terms. On the contrary, a random combinatory term is not strongly normalizing. The main technique used in this section is the theory of generating functions.

As stated in Section \ref{lambda} we can look at combinatory terms as at rooted binary trees whose leaves are labeled with combinators $K$, $S$ and $I$.
We denote by $\mathcal{F}_n$ the number of such trees with $n$ inner nodes
(see Section \ref{Combinatory Logic}).
Obviously the set $\mathcal{F}_n$ is finite. We denote its cardinality by $F_n$. It is trivial to notice that $F_n = C(n) 3^{n+1}$  where $C(n)$ is the $n$-th Catalan number (see Proposition \ref{catalan}).

\begin{prop}\label{prop_gf_e}\hfill
\begin{enumerate}[\em(1)]
\item The generating function $f$ enumerating \correc{cardinality of} the set of combinators (sequence $F_n$) is given by
$$f (z)= \frac{1-\sqrt{1-12z}}{2z}.$$
\item Let $t_0 \in \mathcal{F}_{n_0}$ be a combinator of size $n_0 \geq 1$ . The generating function $f_{t_0}$ enumerating \correct{cardinality of} the set of all combinators having $t_0$ as a subterm is given by
$$f_{t_0} (z) = \frac{-\sqrt{1 - 12z} + \sqrt {1 - 12z + 4z^{n_0 + 1}} }{2z}.$$
\end{enumerate}
\end{prop}

\proof\hfill
\begin{enumerate}[(1)]
\item $F_n$ denotes the number of combinators of size $n$. Since there are three combinators of size $0$, we have $F_0=3$. Combinators of size $n \geq 1$ are built of two combinators of sizes $i$ and $n-i-1$ ($i=0,\ldots,n-1$), \correct{respectively,} thus $F_n=\sum_{i=0}^{n-1}F_i F_{n-i-1}$. From this recurrence relation we obtain that the generating function $f$ for the sequence $(F_n)$ satisfies the equation
$$f (z)= 3 + z (f (z))^2.$$
Solving this equation in $f(z)$ we get two solutions:
$$\frac{1-\sqrt{1-12z}}{2z} \quad \text{and} \quad \frac{1+\sqrt{1-12z}}{2z}.$$
We have $F_0=3$, so $\lim_{z \to 0}f(z)=3$. Thus, the desired generating function is given by the first solution.

\item  Let $t$ be a combinator having $t_0$ as a subterm. Then either $t$ is equal to $t_0$ or \correct{$t$} is of the form of application  $t= t_1 \; t_2$ in which case either $t_0$ is a subterm of $t_1$ but not of $t_2$ or $t_0$ is a subterm of $t_2$ but not of $t_1$ or, finally, $t_0$ is a subterm of both $t_1$ and $t_2$. We get the following equation:
$$ f_{t_0}(z) = z^{n_0} + 2z f_{t_0} (z) \left( f(z) - f_{t_0}(z) \right) + z(f_{t_0}(z))^2 ,$$
which  can be simplified to
$$ f_{t_0}(z)  = z^{n_0} + 2z f_{t_0}(z) f(z) - z(f_{t_0}(z))^2 .$$
Solving this equation in $f_{t_0}$ gives us two possible solutions:
$$\frac{-\sqrt{1 - 12z} + \sqrt {1 - 12z + 4z^{n_0 + 1}} }{2z} \quad \text{and} \quad \frac{-\sqrt{1 - 12z} - \sqrt {1 - 12z + 4z^{n_0 + 1}} }{2z} .$$
Since $n_0 \geq 1$, there is no term of size $0$ having $t_0$ as a subterm. Thus, $\lim_{z \to 0}f_{t_0}(z) = 0$. The first function satisfies this condition, so this is the wanted generating function. \qed\smallskip
\end{enumerate}

\noindent The following theorem shows that the result similar to Theorem \ref{avoidaux} is not valid in combinatory logic.

\begin{thm}\label{main-cl}
Let $t_0$ be a combinator. The density of combinators having $t_0$ as a subterm is $1$.
\end{thm}

\proof
We prove this result applying Theorem \ref{glowniejsze}. We start \correc{by} normalizing \correc{the} functions $f_{t_0}$ and $f$ in such a way that the \correc{closest singularity to the origin is located at $z=1$}. Hence, we define functions $\overline{f_{t_0}} (z) := z f_{t_0} (z/12)$ and $\overline{f} (z) := z f (z/12)$. We get
$$\overline{f_{t_0}}(z) = - \frac{\sqrt {1-z}}{2} + \frac{\sqrt {1-z + 4\, \left( {\frac {z}{12}} \right) ^{n_0 + 1}}}{2}, \qquad
\overline{f} (z) = \frac {1}{2} - \frac {1}{2}\,\sqrt {1 - z}.$$
Since $\frac{\sqrt {1-z + 4\, \left( {\frac {z}{12}} \right) ^{n_0 + 1}}}{2}$ is analytic for $|z|\leq1$, the representation above reveals that the only singularity of $\overline{f_{t_0}} (z)$ and $\overline{f} (z)$ located in $|z|\leq1$ is indeed at $z=1$ and both functions $\overline{f_{t_0}}$ and $\overline{f_{t_0}}$ have expansions in the vicinity of $z=1$ of forms $\sum_{p \geq 0} v_p(1-z)^{p/2}$ and $\sum_{p \geq 0} w_p(1-z)^{p/2}$, respectively, with $w_1 = -1/2 \neq 0$. We have to remember that the multiplication by $z$ and the change of the radius of convergence for functions $f_{t_0}$ and $f$ affect sequences represented by the new functions. Therefore, $\overline{f_{t_0}}$ and $\overline{f}$ enumerate sequences $(12)^{1-n} \left( [z^{n-1}] \{f_{t_0} (z) \}\right)$ and $(12)^{1-n} \left([z^{n-1}]\{{f}(z)\}\right)$, respectively.

Now, let us consider functions $\widetilde{f}$ and $\widetilde{f_{t_0}}$ satisfying the following equations:
 $\widetilde{f}(\sqrt {1 - z} )= \overline{f} (z) $ and $\widetilde{f_{t_0}} (\sqrt {1 - z}) = \overline{f_{t_0}} (z) $. They are defined in the following way:
$$ \widetilde{f_{t_0}}(z) = -\frac{z}{2}  +  \frac{\sqrt {{z}^{2}+4\, \left( {\frac{1-{z}^{2}}{12}} \right) ^{n_0 + 1}}}{2}, \qquad
 \widetilde{f} (z) = \frac{1}{2} -  \frac{1}{2} z.$$
By analyticity of functions $(\widetilde{f_{t_0}})'$ and $(\widetilde{f})'$ for $|z|<1$, their derivatives in this circle exist and are as follows:
$$(\widetilde{f_{t_0}})' (z) = -\frac{1}{2} + {\frac {\left( 2\,z-\frac{8}{12} (n_0 + 1) z\, \left( {\frac {1-{z}^{2}}{12}} \right) ^{n_0} \right)}{ 4 \sqrt {{z}^{2 }+4\, \left({\frac {1-{z}^{2}}{12}} \right) ^{n_0 + 1}}}}, \qquad (\widetilde{f})' (z) = -\frac {1}{2}.$$
Finally, by computing the values of those derivatives at $z=0$ we get $(\widetilde{f_{t_0}})' (0) = -\frac{1}{2}$ and $(\widetilde{f})' (0) = -\frac{1}{2}$.

To complete the proof we apply Theorem \ref{glowniejsze}, obtaining:
\[\liminfty{\frac{[z^n]\{f_{t_0} (z)\}}{[z^n]\{f(z)\}}}  = \liminfty{\frac{ (12)^{1-n} [z^{n-1}] \{\correct{\overline{f_{t_0}}} (z)\}}{(12)^{1-n} [z^{n-1}]\{ \correct{\overline{f}}(z)\}}}  = \frac{(\widetilde{ f_{t_0}  })'(0)}{(\widetilde{f})'(0)} = 1 .\eqno{\qEd}\]

\begin{main_theorem}\label{main_CL}
The density of non-strongly normalizing combinators is $1$.
\end{main_theorem}

\proof
Let $\Omega = S \ I \ I \ (S \ I \ I)$. The combinator $\Omega$ reduces to itself and thus is not strongly normalizing. The thesis follows directly from Theorem \ref{main-cl}, since the density of combinators containing $\Omega$ as a subterm is $1$.
\qed 

\section{Discussion}\label{size}

The difference between Theorem \ref{main} in the
$\lambda$-calculus and Theorem \ref{main_CL} in combinatory logic
may be surprising since there are translations between these
systems which respect many properties (including strong normalization). However, these translations do not preserve the size.

The usual translation, which we denote by $T\!_1$, from
combinatory logic to $\lambda$-calculus, is linear: there is a
constant $k$ such that, for all term $t$, $\size(T\!_1 (t)) \leq k \size(t)$. Note that this translation is far from being surjective:
its image has density 0. The usual translation $ T\!_2 $ in the
other direction (see \cite{BAR84}) is not linear but exponential. As far as we
know, $\size(T\!_2 (t))$ is of order
$3^{\size(t)}$.
The point is that $T\!_2$ has to code the variable binding in some way
and this requires the use of many combinators.

\subsection{Future work and open questions}\label{future}
%\vspace{-0.3cm}
 We \correct{present} here some questions % for which it will be
% desirable to have an answer
\corr{left open}.

\begin{enumerate}[(1)]
\item Give the asymptotics of $L_n$ or, at least, better
upper and lower
  bounds.

\item Give the density of typable terms. Numerical
  experiments done by Jue Wang (see \cite{wang}) seem to show that this
  density is 0 \corr{for simple types}.

\item Compute the densities  of strongly normalizing terms
with other notions of size, mainly by changing the size of variables,
and eventually making it non constant.\corr{For what notions of size do we
get a density $1$ as in Theorem \ref{main} or a density $1$ as in
Theorem \ref{main_CL}? Are they sizes for which the density is
neither $0$ nor $1$?}
\end{enumerate}

\subsection{Possible applications}

It is now popular to test programs written in functional languages
using randomly generated % \corr{sets} of
inputs \cite{qcheck}.
For higher-order functional programs where inputs are functions,
% or algorithms for computing functions
this also means the ability to generate typical functions under certain known distributions.

For many typed languages such as OCaml or Haskell, functional programs \corr{can be tested by} supplying random typed $\lambda$-terms generated in compliance with their natural distribution (probably different for different \correct{types} of programs).

For untyped languages such as LISP, \correc{the} problem of testing programs is very close
to \correc{the} capability of generating   pure random $\lambda$-terms.  In our case,
those terms automatically enjoy important properties such as strong
normalization, if they do not use recursive definition\correc{s}.
However, it would be nice to have a distribution where terms with other computationally good properties have density~$1$.

\bigskip

% In light of our results, the distribution induced by the size of terms in
% combinatory logics is dramatically different because most term\correc{s} are not strongly
% normali\correc{sing}, while the distribution \correc{of} pure $\lambda$-term\correc{s} with
% variables of size $0$ enjoys strong normalization \correc{with density $1$}.
One could argue that width at most $2$ is a negative result \correc{since it \correct{shows} that random terms do not contain any tuple of more than $2$ functions, whereas 'natural' programs do contain such kind of subterms.}

% for testing (tuple\correc{s} of functions are common and \correc{are} encoded with terms
% of arbitrary width).

\bigskip

\corr{Anyway,} results \corr{and methods} presented in this paper \corr{can be used as} the starting point for \correc{further} research based on other notion\correc{s} of size \corr{which are meaningful} for applications.

%Possible
%direction\correc{s} of research \correc{are} discussed in the next section.
%\appendix
%\input{appendix.tex}

%Thanks
\section*{Acknowledgments}
We would like to thank the anonymous referees for their numerous, precise and
fruitful remarks.


\begin{thebibliography}{}
%\vspace{-0.3cm}

\bibitem{BAR84} H. Barendregt, \emph{The Lambda Calculus: Its Syntax and Semantics}.
  Studies in Logic and The Foundations of Mathematics, vol. 103, North-Holland, 1984.

\bibitem{theyssier} L. Boyer, G. Theyssier,
\emph{On Local Symmetries and Universality in Cellular Automata}. 26th International Symposium on Theoretical Aspects of Computer Science (STACS), Dagstuhl Seminar Proceedings, 2009,\\\texttt{http://stacs2009.informatik.uni-freiburg.de/proceedings.php}

\bibitem{qcheck} K. Claessen, J. Hughes, \emph{QuickCheck: A Lightweight Tool for Random Testing of Haskell Programs}.
Proc. of International Conference on Functional Programming (ICFP), ACM SIGPLAN, pp. 268-279, 2000.

%\bibitem{CFGG04} B. Chauvin, P. Flajolet, D. Gardy and B. Gittenberger.
%And/Or trees revisited,
%\emph{Combinatorics, Probability and Computing},
%13(4-5):475-497, 2004.

%\bibitem{Comtet} L. Comtet, \emph{Advanced combinatorics.
%         The art of finite and infinite expansions}.
%         Revised and enlarged edition,
%         Reidel, Dordrecht, 1974.

%\bibitem{CFGG04} B. Chauvin, P. Flajolet, D. Gardy and B. Gittenberger.
%And/Or trees revisited,
%\emph{Combinatorics, Probability and Computing},
%13(4-5):475-497, 2004.

%\bibitem{FGOR93} {Philippe Flajolet and
%               Zhicheng Gao and
%                Andrew M. Odlyzko and
%                L. Bruce Richmond,
%\emph{The Distribution of Heights of Binary Trees and Other Simple
%Trees}, volume 2 of \emp h{Combinatorics, Probability {\&}
%Computing}. Cambridge University Press, 1993. }

\bibitem{curry_feys} H.B. Curry, R. Feys, \emph{Combinatory Logic}. Vol. I. Amsterdam: North Holland, 1958.

\bibitem{david01} R. David, \emph{Normalization without reducibility}. APAL 107 (2001), pp. 121-130.

\bibitem{davidWeb} R. David, \emph{A short proof of the strong normalization of the simply typed lambda calculus}.\\\texttt{http://www.lama.univ-savoie.fr/\~{}david/}

\bibitem{fs01} P. Flajolet, R. Sedgewick, \emph{Analytic combinatorics}. Cambridge University Press, 2008.
%: functional equations, rational and
%algebraic functions,} INRIA, Number 4103, 2001. Also  see the web
%page
%  http://algo.inria.fr/flajolet/Publications/books.html

\bibitem{FGGZ} H. Fournier, D. Gardy, A. Genitrini, M. Zaionc, \emph{Classical and intuitionistic logic are asymptotically identical}. Computer Science Logic 2007, Lecture Notes in Computer Science 4646, pp. 177-193.

%\bibitem{gardy-dmtcs} D. Gardy.
%Random Boolean expressions,
%\emph{Colloquium on Computational Logic and Applications},
%Chamb\'ery (France), June 2005.
%In \emph{Discrete Mathematics and Theoretical Computer Science} proceedings AF, pp 1-36, 2006.

%\bibitem{GW05} D. Gardy, A. Woods. And/or tree
%probabilities of Boolean function,
%\emph{Discrete Mathematics and Theoretical Computer Science},pp 139-146, 2005.

\bibitem{GK-09} A. Genitrini, J. Kozik, \emph{Quantitative comparison of Intuitionistic and Classical logics -- full propositional system}.
     LFCS09, Lecture Notes in Computer Science 5407, pp. 280-294, 2009.

\bibitem{GKZ} A. Genitrini, J. Kozik, M. Zaionc,
 \emph{Intuitionistic vs. Classical Tautologies, Quantitative Comparison}. Lecture Notes in Computer Science 4941, pp. 100-109, 2008.

\bibitem{hamkins} J.D. Hamkins and A. Miasnikov,
\emph{The halting problem is decidable on a set of asymptotic probability one}. Notre Dame J. Formal Logic 47(4), pp. 515-524, 2006.

\bibitem{kos-zaionc03} Z. Kostrzycka, M. Zaionc,
\emph{Statistics of intuitionistic versus classical logic}. Studia Logica, 76(3), pp. 307-328, 2004.

%\bibitem{sorensen} M.H.B S{\o}rensen,
%\emph{Normalization in $\lambda$-calculus and Type theory,} PhD
%thesis University of Copenhagen.
%ftp://ftp.diku.dk/diku/semantics/papers/D-367.ps.gz


%\bibitem{LS97} H. Lefmann and P. Savick\'y. Some typical
%properties of large {And/Or} {B}oolean formulas, \emph{Random Structures and Algorithms},
%vol 10, pp 337-351, 1997.

%\bibitem{mat05} G. Matecki. Asymptotic density for equivalence,
%\emph{Electronic Notes in Theoretical Computer Science}, 140:81-91, 2005.

\bibitem{szego} G. Szeg\"{o}, \emph{Orthogonal polynomials}. American Mathematical Society Colloquium Series Publication, 1967.

\bibitem{mtz00} M. Moczurad, J. Tyszkiewicz, M. Zaionc, \emph{Statistical properties of simple types}. Mathematical Structures in Computer Science, 10(5), pp. 575-594, 2000.

%\bibitem{urzy2006} M. S{\o}rensen and P. Urzyczyn.
%\emph{Lectures on the Curry-Howard Isomorphism}, volume 149 of
%\emph{Studies in Logic and the Foundations of Mathematics}.
%Elsevier Science, 2006.

%\bibitem{Sz} G. Szeg\"o, \emph{Orthogonal polynomials},
%         fourth ed., AMS, Colloquium Publications, 23,
%         Providence, 1975.

\bibitem{regnier} L. Regnier, \emph{Une \'equivalence sur les
    lambda-termes}. Theoretical Computer Science, Volume 126(2), pp. 281-292, 1994.

\bibitem{rybalov} A. Rybalov,
\emph{On the strongly generic undecidability of the Halting Problem}. Theoretical Computer Science, Volume 377, Issues 1-3, pp. 268-270, 31 May 2007.

\bibitem{schonfinkel} M. Sch\"{o}nfinkel, \emph{\"{U}ber die Bausteine der mathematischen Logik}. Mathematische Annalen 92, pp. 305-316, 1924.

\bibitem{smullyan} R.M. Smullyan, \emph{To Mock a Mockingbird and Other Logic Puzzles: Including an Amazing Adventure in Combinatory Logic}. Knopf, 1985.

\bibitem{wang} J. Wang, \emph{Generating Random Lambda Calculus Terms}.\\ \texttt{http://cs-people.bu.edu/juewang/research.html}

\bibitem{Wilf} H.S. Wilf, \emph{Generatingfunctionology}. Second ed., Academic Press, Boston, 1994.

\bibitem{zaionc05} M. Zaionc, \emph{ On the asymptotic density of tautologies in logic of implication and negation}. Reports on Mathematical Logic, vol 39,
pp. 67-87, 2005.

\bibitem{zai06} M. Zaionc, \emph{Probability distribution for simple tautologies}. Theoretical Computer Science, 355(2), pp. 243-260, 2006.

\end{thebibliography}
\end{document}